\documentclass[11pt]{amsart}
\usepackage{amsmath,amssymb,mathtools}
\usepackage[margin=1in]{geometry}
\usepackage{needspace}
\usepackage{microtype}
\usepackage{xcolor}
\usepackage[colorlinks=true,linkcolor=blue!45!black,citecolor=blue!45!black,urlcolor=blue!45!black]{hyperref}
\DeclareSymbolFont{averagedintegrals}{U}{esint}{m}{n}
\DeclareMathSymbol{\fintop}{\mathop}{averagedintegrals}{'037}
\newcommand{\fint}{\fintop\nolimits}
\newcommand{\R}{\mathbb{R}}
\newcommand{\cR}{\mathcal{R}}
\newcommand{\cS}{\mathcal{S}}

\newcommand{\Id}{\mathrm{Id}}
\newcommand{\Test}{\mathrm{Test}}
\newcommand{\pmGH}{\mathrm{pmGH}}

\newcommand{\Leb}{\mathcal{L}}
\newcommand{\Ch}{\mathrm{Ch}}
\newcommand{\RCD}{\mathrm{RCD}}

\newcommand{\supp}{\operatorname{supp}}
\newcommand{\Lip}{\operatorname{Lip}}
\newcommand{\essinf}{\operatorname*{ess\,inf}}
\DeclareMathOperator{\Ric}{Ric}
\DeclareMathOperator{\vol}{vol}
\DeclareMathOperator{\Hess}{Hess}
\DeclareMathOperator{\tr}{tr}
\DeclareMathOperator{\diver}{div}

\theoremstyle{plain}
\newtheorem{thm}{Theorem}[section]
\newtheorem{prop}[thm]{Proposition}
\newtheorem{lem}[thm]{Lemma}

\theoremstyle{definition}

\theoremstyle{remark}
\newtheorem{rem}[thm]{Remark}
\numberwithin{equation}{section}

\title{The regular set of Ricci limit spaces}
\author{Tianyin Ren}
\address{Department of Mathematics, Huaqiao University}
\email{tianyin.ren@hqu.edu.cn}
\date{}
\begin{document}
\begin{abstract}
	Let $(X,d,p)$ be a nontrivial pointed Gromov--Hausdorff limit of complete connected Riemannian $n$-manifolds satisfying $\Ric\geq-(n-1)g$. For $1\leq j\leq n$, set
	\[ \cR_j(X) = \bigl\{x\in X:\text{ every tangent cone at } x \text{ is isometric to } \R^j\bigr\}, \qquad \cR(X)=\bigsqcup_{j=1}^n\cR_j(X). \]
	We prove that there is a unique $k$ such that $\cR(X)=\cR_k(X)$, answering the pointwise regular-dimension question originating in Cheeger and Colding's work~\cite[p408]{CC97}. We construct a volume-corrected heat entropy on RCD space and identify the distributional logarithmic-density defect $\tr(\Hess f)-\Delta f$ by Green-volume regularization.
\end{abstract}
\maketitle

\section{Introduction}

Let $(M_i^n,g_i,p_i)$ be a sequence of pointed complete connected $n$-dimensional Riemannian manifolds such that
\[ \Ric_{g_i}\geq-(n-1)g_i, \qquad (M_i,d_{g_i},p_i)\xrightarrow{\mathrm{pGH}}(X,d,p). \]
The limit $(X,d,p)$ is called a \emph{pointed Ricci limit space}.

For $x\in X$ and a decreasing sequence $r_i\to0^+$, any pointed Gromov--Hausdorff subsequential limit
\[ (X,r_i^{-1}d,x)\xrightarrow{\mathrm{pGH}}(T_{x}X,d_\infty,z_x) \]
is called a \emph{tangent cone} at $x$. The tangent cone may depend on the sequence of scales. A point $x$ is \emph{strongly $j$-regular} if every tangent cone at $x$ is isometric to $(\R^j,0)$. We write
\[ \cR_j(X) = \bigl\{x\in X:\text{ every tangent cone at } x \text{ is isometric to } \R^j\bigr\}, \qquad \cR(X)=\bigsqcup_{j=1}^n\cR_j(X). \]
We call $\cR(X)$ the regular set. Throughout the paper, regularity is understood in this strong sense; the existence of a single Euclidean tangent cone is not sufficient.

Let $l =\dim_{\mathcal H}X$ be the Hausdorff dimension of $X$ and $\cS(X)=X\setminus\cR(X)$. Cheeger and Colding~\cite[p408]{CC97} expressed the hope that $\cR(X)$ is connected, $\cR(X)=\cR_l(X)$, and $\dim_{\mathcal H}\cS(X)\leq l -1$, and more generally that $\cS(X)$ has codimension at least one with respect to any natural measure assigning positive measure to $\cR(X)$. In the noncollapsed case, $l = n$, these properties hold, with the stronger bound $\dim_{\mathcal H}\cS(X)\leq n-2$.

Cheeger and Colding~\cite{CC97,CC00a,CC00b}, established the rectifiable structure of Ricci limit spaces. For any renormalized limit measure $m$, the regular set has full $m$-measure. Moreover, each stratum $\cR_j(X)$ can, up to a $m$-null set, be covered by countably many Borel sets bi-Lipschitz to subsets of $\R^j$, with $m \ll\mathcal H^j$ on these sets~\cite{CC00b}.

Colding and Naber~\cite{CN12} proved that there is a unique integer $k$ such that $\cR_k(X)$ has full $m$-measure. Combined with Cheeger--Colding's rectifiability theorem, this implies that $(X,d,m)$ is countably $k$-rectifiable and $m \ll\mathcal H^k$. This $k$ is the essential (or rectifiable) dimension of $X$. Colding and Naber also proved that $\cR_k(X)$ is weakly convex~\cite{CN12}. Consequently, $\cR_k(X)$ is connected and dense in $X$, and $\cR(X)$ is connected as well, since $\cR_k(X)\subseteq\cR(X)\subseteq\overline{\cR_k(X)}=X$.

Kapovitch and Li~\cite{KL18} showed that $k$ is the largest dimension of a nonempty regular stratum. It remains to exclude $\cR_j(X)$ for $j<k$.

Pan and Wei~\cite{PW22} constructed collapsing Ricci limit spaces with $\cR(X)=\cR_2(X)$ and $\dim_{\mathcal H}\cS(X)>2$, whose Hausdorff dimension can be nonintegral. 
Thus the regular dimension need not agree with the Hausdorff dimension of the space. The question whether all regular points have the same dimension nevertheless remained open. These results do not exclude regular points of a different dimension on a set of measure zero. 

Several low-dimensional cases are known. Chen proved that $\cR_1(X)\neq\varnothing$ forces $X$ to be a one-dimensional topological manifold~\cite{Chen16}.  In~\cite{Chen24}, a $2$-regular point in the interior of a geodesic implies $2$-rectifiability. Our main result establishes pointwise constancy in arbitrary regular dimension, as follows.

\begin{thm}\label{thm:ricci-regular-dimension}
	Let $(X,d,p)$ be a nontrivial pointed Ricci limit space of complete connected $n$-dimensional Riemannian manifolds with $\Ric\geq-(n-1)g$. Then there is a unique integer $k\in\{1,\ldots,n\}$ such that
	\begin{equation}
		\varnothing\neq\cR(X)=\cR_k(X).
	\end{equation}
\end{thm}

We describe how entropy enters the proof of Theorem~\ref{thm:ricci-regular-dimension}. We work on a complete connected $\RCD(-K,N)$ space. Throughout the paper, $k$ denotes the essential dimension of $X$, while $n$ denotes the dimension of the approximating manifolds. For $k\geq2$, we may increase the finite synthetic upper dimension parameter so that $N>k$.

Let $k_t^y$ denote the positive heat-kernel function of the reference-measure heat semigroup $P_t$ with pole $y$, so that $P_tf(y)=\int_X f k_t^y\,dm$ and $k_t^y(\xi)=k_t^\xi(y)$. Put $u_t^y=-\log k_t^y$. Define the entropy and Fisher information by
\[ E_y(t)=\int_X k_t^y\log k_t^y\,dm, \qquad I_y(t)=\int_X|\nabla u_t^y|^2k_t^y\,dm. \]
The entropy $-E_y$ was used by Nash~\cite{Nash58}. Its normalization by the essential dimension is
\[ \mathcal F_y^0(t)=-E_y(t)-\frac k2\log t. \]
If $z\in\cR_j(X)$ with $j<k$, the volume ratios and Gaussian bounds give $E_z(t)=\frac j2\log(1/t)+o(\log(1/t))$. At each fixed ordinary $k$-regular point $y$, we have
\[ \mathcal F_y^0(t)=o_y(\log(1/t)) \qquad(t\to0^+). \]
The remainder $o_y$ may depend on $y$: this is a pointwise asymptotic with no uniform control near $z$. Our comparison requires ordinary centers $y_t\to z$ chosen with $t$, so substituting $y=y_t$ into this estimate is not justified.

Ni~\cite{Ni04}, motivated by Perelman's entropy for the Ricci flow~\cite{Per02}, introduced the corresponding $W$-entropy for the heat equation on a fixed Riemannian manifold. We use the normalization
\[ \mathcal W_y(t)=tI_y(t)+\mathcal F_y^0(t)=tI_y(t)-E_y(t)-\frac k2\log t. \]
In the unweighted smooth case, $\mathcal W_y-\frac k2\log(4\pi)-k$ is Ni's entropy. On smooth weighted manifolds, Li~\cite{Li12} and Li--Li~\cite{LL15} established corresponding $W$-entropy formulas for the Witten Laplacian.

On $\RCD(0,N)$ spaces, Kuwada--Li~\cite{KLi21} proved monotonicity of the $W$-entropy normalized by the synthetic dimension $N$. Li--Zhang~\cite{LZ25} also study rank-dependent entropy formulas, under their stated Riemannian Bochner and heat-flow integrability hypotheses. The estimates below retain the trace defect and justify the heat-kernel approximation needed for the normalization by $k$ and the local density correction.

To control this dependence on $y$, we introduce a volume-adjusted entropy $\mathcal F$ using the intrinsic density. Define, almost everywhere,
\[ \theta_k(x)=\lim_{r\to0^+}\frac{m(B_r(x))}{\omega_k r^k}, \qquad \Phi(x)=-\log\theta_k(x). \]
Here $\omega_k$ is the volume of the Euclidean unit ball. The density is positive and finite on a full-measure set~\cite{BS18}.

Choose a compactly supported cutoff $\chi$ equal to one near $z$, put $\Psi=\chi\Phi$, and set $C_y(t)=-P_t\Psi(y)$. The volume-corrected entropy is
\[ \mathcal F_y(t)=\mathcal W_y(t)-2C_y(t)=tI_y(t)-E_y(t)-\frac k2\log t-2C_y(t). \]
The following result supplies the local integrability and distributional identity needed to justify and estimate this correction.

\begin{thm}[Logarithmic density]\label{thm:log-density-identity}
	Let $(X,d,m)$ be a complete connected $\RCD(-K,N)$ space with full support and measure finite on bounded sets, where $K\geq0$ and $1<N<\infty$. Suppose its essential dimension is $k\geq1$. Then:
	\begin{enumerate}
		\item[(i)] The logarithmic density $\Phi=-\log\theta_k$ belongs to $L^2_{\mathrm{loc}}(m)$.
		\item[(ii)] For every global test function $f$ and every compactly supported Lipschitz function $\phi$,
		\begin{equation}\label{eq:log-density-identity}
			-\int_X\Phi\,[\phi\Delta f+\langle\nabla \phi,\nabla f\rangle]\,dm = \int_X\phi\,[\tr(\Hess f)-\Delta f]\,dm.
		\end{equation}
	\end{enumerate}
\end{thm}

\begin{rem}
	Identity~\eqref{eq:log-density-identity} is directional and distributional; a full Sobolev gradient for $\Phi$ is not assumed. On a smooth weighted manifold with $m=e^{-\Phi}\vol_g$, it is the familiar relation $\tr_g\Hess f-\Delta f=\langle\nabla\Phi,\nabla f\rangle$. The theorem identifies the intrinsic density defect without an open smooth locus. A subsequent paper will use Theorem~\ref{thm:log-density-identity} to give another proof of Honda's conjectures~\cite{Hon20}, recently resolved in~\cite{BG26}.
\end{rem}

Let $D_t^y=\tr(\Hess u_t^y)-\Delta u_t^y$. The localized form of~\eqref{eq:log-density-identity} gives
\[ C_y'(t)=-\int_X D_t^y k_t^y\,dm+\operatorname{err}_y(t), \]
where the time-integrated cutoff error is exponentially small, uniformly for centers near $z$. The entropy inequality for $\mathcal W_y$ contains the linear term $-2\int_X D_t^y k_t^y\,dm$. Differentiating $-2C_y$ cancels this term and yields a time-propagation upper bound. The resulting estimate has constants uniform in the nearby center and an explicit dependence on $\Phi(y)$, replacing the unspecified point-dependent remainder $o_y(\log(1/t))$.

At the proposed exceptional point $z$, harmonic almost splitting maps give the sharp lower bound $j/4$ for the $t^2$-weighted Hessian energy. This accounts for the entire leading logarithmic coefficient in the integrated entropy dissipation, leaving only a sublogarithmic time integral of the trace defect. Consequently $C_z(t)=o(\log(1/t))$, so the correction preserves the dimensional growth:
\[ \mathcal F_z(t)=\frac{k-j}{2}\log(1/t)+o(\log(1/t)). \]
At almost every ordinary $k$-regular point $y$, the corrected inequality and the density limit instead give $\mathcal F_y(t)\leq B_0+\Phi(y)+O(t)$, with constants uniform near $z$. The sublogarithmic heat average at $z$, Gaussian bounds, and $\Phi\in L^2_{\mathrm{loc}}$ allow us to choose ordinary points $y_t$ with $d(z,y_t)<t^{M_0}$, for a fixed $M_0>(N+1)/2$, and $\Phi(y_t)\leq o(\log(1/t))$. Heat-kernel comparison between $z$ and $y_t$ then yields
\[ \mathcal F_z(t)\leq o(\log(1/t)), \]
contradicting the positive coefficient when $j<k$. This excludes strongly regular points below the essential dimension and, together with the result of Kapovitch--Li, proves Theorem~\ref{thm:ricci-regular-dimension}.

The adjusted entropy also has a useful interpretation and further applications on smooth weighted manifolds. Let $(M^k,g,m)$ satisfy $dm=e^{-\Phi}\,d\vol_g$, with $\Phi$ smooth, so that $\theta_k=e^{-\Phi}$ and $\Delta f=\Delta_g f-\langle\nabla\Phi,\nabla f\rangle$. Suppose first that $M$ is closed and choose $\chi\equiv1$. Then the unlocalized entropy is
\[ \mathcal F_y(t)=tI_y(t)-E_y(t)-\frac k2\log t+2P_t\Phi(y). \]
Let $\Ric_\Phi=\Ric+\Hess\Phi$. The weighted Bochner identity and $(P_t\Phi)'(y)=\int_M D_t^y k_t^y\,dm$ give the exact dissipation formula
\[ \frac{d}{dt}\mathcal F_y(t)=-2t\int_M\left[\left|\Hess u_t^y-\frac{g}{2t}\right|^2+\Ric_\Phi(\nabla u_t^y,\nabla u_t^y)\right]k_t^y\,dm. \]
The same identities hold on complete manifolds when the heat flow is conservative, the terms are integrable, and differentiation and integration by parts are justified. In the smooth setting, this formula follows by density and dimension renormalization from the Witten-Laplacian entropy formulas in~\cite{Li12,LL15}. The formulation above keeps the geometric dimension $k$ and the heat average of the logarithmic density explicit. Our RCD argument makes this correction available for an intrinsically defined density having only the distributional directional regularity supplied by Theorem~\ref{thm:log-density-identity}.

Several consequences follow directly in this smooth setting. If $\Ric_\Phi\geq0$, then $\mathcal F_y$ is nonincreasing. With $c_k=k+\frac k2\log(4\pi)$, the small-time entropy and Fisher asymptotics, which hold in particular on closed manifolds, give
\[ \lim_{t\to0^+}\bigl(\mathcal F_y(t)-c_k\bigr)=\Phi(y), \qquad \mathcal F_y(t)\leq c_k+\Phi(y). \]
Thus the adjusted entropy recovers the local measure density from the heat flow and gives a heat-kernel entropy bound in the geometric dimension. Moreover, if $\Ric_\Phi^N=\Ric_\Phi-\nabla\Phi\otimes\nabla\Phi/(N-k)\geq0$ for some $N>k$, then for $0<a<b$ the same dissipation formula gives
\[ \int_a^b t\int_M\left[\left|\Hess u_t^y-\frac{g}{2t}\right|^2+\frac{|D_t^y|^2}{N-k}\right]k_t^y\,dm\,dt\leq\frac{\mathcal F_y(a)-\mathcal F_y(b)}2. \]
Small entropy loss therefore controls both the Gaussian Hessian defect and the trace defect. This supplies a quantitative analytic tool for weighted rigidity and stability questions; extending such control to geometric cone rigidity under collapsing limits is a further problem.

We now sketch the proof of Theorem~\ref{thm:log-density-identity}. Work on $Y=X\times\R^3$, with product measure $m_Y=m\times\Leb^3$ and essential dimension $\ell=k+3$, and choose a killed Green kernel $G$. The product ensures $\ell>2$, so the Green kernel has a singularity; the exponential killing factor makes its large-time integral finite. At almost every pole, Euclidean blow-up and the unit Dirac normalization determine the exact coefficient of this singularity. For a smooth nondecreasing cutoff $H$ that vanishes on $[0,1]$ and equals one on $[2,\infty)$, define the regularized superlevel volume $V_\epsilon(a)=\int_Y H(\epsilon G(a,b))\,dm_Y(b)$. The Green asymptotic gives
\[ V_\epsilon((x,0))=A_{H,\ell}\epsilon^{\ell/(\ell-2)}\theta_k(x)^{-2/(\ell-2)}(1+o(1)), \]
where $A_{H,\ell}>0$ is an explicit Euclidean constant. Thus the normalized logarithms
\[ \Phi_\epsilon(x)=\frac{\ell-2}{2}\log\left(\frac{V_\epsilon((x,0))}{A_{H,\ell}\epsilon^{\ell/(\ell-2)}}\right) \]
are locally Lipschitz on each relatively compact region for sufficiently small $\epsilon$, and converge to $\Phi$ almost everywhere and locally in measure.

For part~(i), first-order Green estimates and local maximal-function bounds give uniform local $L^2$ control of $\langle\nabla\Phi_\epsilon,\nabla f\rangle$ for each compactly supported test function $f$, in terms of its Hessian, Laplacian, gradient, and function norms. To turn this directional estimate into a bound for $\Phi_\epsilon$ itself, we apply local transposition separately to smooth positive and negative parts. Riesz representation separates a uniformly bounded $L^2$ term from a harmonic remainder. Convergence in measure anchors that remainder on a set of positive measure, and local harmonic estimates and Harnack's inequality bound it on a smaller ball. This yields uniform local $L^2$ bounds for $\Phi_\epsilon$, hence $\Phi\in L^2_{\mathrm{loc}}$ and $\Phi_\epsilon\rightharpoonup\Phi$ locally in $L^2$.

For part~(ii), passing to the limit in integration by parts first defines unique local $L^2$ directional representatives $b_f$ satisfying
\[ \int_X\phi b_f\,dm=-\int_X\Phi\bigl[\phi\Delta f+\langle\nabla\phi,\nabla f\rangle\bigr]\,dm \]
for compactly supported Lipschitz $\phi$ in the region under consideration.

We then identify $b_f$ by an exact formula at each regularization scale. A two-endpoint integration-by-parts identity for truncated Green kernels expresses the derivative of $V_\epsilon$ through the tensor $2\Hess^Y v-(\Delta^Y v)g_Y$, where $v$ is a compactly supported extension of $f$ to $Y$. The second potential in this formula is the resolvent of the shell forcing $\epsilon H'(\epsilon G)$. At almost every pole, annular estimates control the singular and outer tails, while rescaling to the Euclidean tangent cone freezes the tensor coefficients and gives strong Sobolev convergence of the potentials on fixed annuli. Rotational invariance then evaluates the normalized stress, and the killing term vanishes. Combining this limit with the normalized superlevel average of $\Delta^Y v$ yields $\tr(\Hess f)-\Delta f$. The uniform directional $L^2$ estimate identifies this limit with the representative just constructed:
\[ b_f=\tr(\Hess f)-\Delta f \qquad m\text{-almost everywhere}. \]
Substitution into its distributional representation proves~\eqref{eq:log-density-identity}; a final cutoff argument extends the identity from compactly supported test functions to all global test functions. These two parts are proved in the final two sections.

The paper is organized as follows. Section~\ref{sec:heat-quantities} establishes the heat identities and trace-defect control, including the facts about measured tangent cones needed at strongly regular points. Section~\ref{sec:entropy-exclusion} develops the corrected entropy estimates and proves Theorem~\ref{thm:ricci-regular-dimension} using Theorem~\ref{thm:log-density-identity}. The proof of Theorem~\ref{thm:log-density-identity} occupies the final two sections: Section~\ref{sec:green-asymptotics} proves part~(i) by Green-volume regularization and local transposition, and Section~\ref{sec:stress-limit} proves part~(ii) through the Green stress limit.

\section{Heat quantities and the trace defect}\label{sec:heat-quantities}
For the heat estimates in this section, $(X,d,m)$ is a complete connected $\RCD(-K,N)$ space with full support, measure finite on bounded sets, and essential dimension $k$, where $K\geq0$ and $2\leq k<N<\infty$. Lemma~\ref{lem:canonical-measured-tangent} below is stated in the general range $1<N<\infty$ and does not require $k\geq2$.

\subsection{Notation and Gaussian bounds}
We recall the differential operators used below, following~\cite{Che99,Gig18}. The Sobolev space $W^{1,2}(X)$ consists of $L^2(m)$ functions with minimal $2$-weak upper gradient $|Df|\in L^2(m)$. The $\RCD$ condition makes the Cheeger energy $\Ch(f)=\frac12\int_X|Df|^2\,dm$ quadratic. The differential $df$ belongs to the cotangent module $L^2(T^*X)$, and its metric dual is the gradient $\nabla f\in L^2(TX)$, with $|\nabla f|=|df|=|Df|$ almost everywhere. The pointwise inner product on the tangent module is denoted by $\langle\cdot,\cdot\rangle$.

The reference-measure Laplacian is defined by integration by parts: $f\in D(\Delta)$ if $f\in W^{1,2}(X)$ and there is $h\in L^2(m)$ such that
\[ \int_X\varphi h\,dm=-\int_X\langle\nabla\varphi,\nabla f\rangle\,dm \qquad\text{for every }\varphi\in W^{1,2}(X). \]
We then set $\Delta f=h$. Thus $\Delta=\diver_m\nabla$ is nonpositive on $L^2(m)$, and the heat semigroup satisfies $\partial_tP_t=\Delta P_t$. We use the test class
\[ \Test(X)=\{f\in D(\Delta)\cap L^\infty: |\nabla f|\in L^\infty,\ \Delta f\in W^{1,2}\cap L^\infty\}. \]

The Hessian is a weak covariant $2$-tensor. More precisely, $f\in W^{2,2}(X)$ if $f\in W^{1,2}(X)$ admits a tensor $\Hess f\in L^2((T^*X)^{\otimes2})$ satisfying
\begin{align*}
2\int_X h\,\Hess f(\nabla g_1,\nabla g_2)\,dm
&=-\int_X\langle\nabla f,\nabla g_1\rangle\,\diver_m(h\nabla g_2)\,dm\\
&\quad-\int_X\langle\nabla f,\nabla g_2\rangle\,\diver_m(h\nabla g_1)\,dm\\
&\quad-\int_X h\,\langle\nabla f,\nabla\langle\nabla g_1,\nabla g_2\rangle\rangle\,dm
\end{align*}
for every $g_1,g_2\in\Test(X)$ and bounded Lipschitz $h$, where
\[ \diver_m(h\nabla g)=h\Delta g+\langle\nabla h,\nabla g\rangle. \]
This tensor is unique and symmetric; moreover, $D(\Delta)\subset W^{2,2}(X)$. Local gradients, Laplacians, and Hessians are defined by the same identities with compactly supported tests, and the operators below are understood locally whenever the functions are not globally Sobolev.

Hessian traces and norms are taken in the reference tangent module, whose almost-everywhere rank is the essential dimension $k$~\cite{BS18}. On each piece of a measurable partition admitting a local orthonormal frame $e_1,\ldots,e_k$, we have
\[ \tr(\Hess f)=\sum_{\alpha=1}^k\Hess f(e_\alpha,e_\alpha), \qquad |\Hess f|^2=\sum_{\alpha,\beta=1}^k\bigl|\Hess f(e_\alpha,e_\beta)\bigr|^2. \]

Unlabelled differential operators act on $X$. Operators on $Y=X\times\R^3$ carry a superscript $Y$. Superscripts $i$ and $\mathrm E$ indicate the rescaled metric $d_i$ and the Euclidean limit, respectively. Subscripts $a,b,z$ on kernel derivatives indicate the argument being differentiated.

We use the Gaussian kernel and gradient estimates of Jiang--Li--Zhang~\cite{JLZ16}. For a point $y$, set
\[ u_t^y=-\log k_t^y,  \qquad D_t^y=\tr(\Hess u_t^y)-\Delta u_t^y, \]
\[ E_y(t)=\int_{X}k_t^y\log k_t^y\,dm, \qquad I_y(t)=\int_X|\nabla u_t^y|^2k_t^y\,dm, \]
\[  A_y(t) = \int_X \left[ |\Hess u_t^y|^2+\frac{{(D_t^y)}^2}{N-k} \right]k_t^y\,dm. \]

Gaussian estimates have the form
\begin{align}
	\frac{C^{-1}}{V_t(y)} \exp\left(-\frac{d{(y,z)}^2}{3t}-ct\right) & \leq k_t^y(z) \leq \frac{C}{V_t(y)} \exp\left(-\frac{d{(y,z)}^2}{5t}+ct\right), \notag \\
	|\nabla k_t^y|(z)                                             & \leq \frac{C}{\sqrt t\,V_t(y)} \exp\left(-\frac{d{(y,z)}^2}{5t}+ct\right),
	\label{eq:entropy-gaussian}
\end{align}
where $V_t(y)=m(B_{\sqrt t}(y))$. Annular summation gives, on every bounded positive time interval,
\begin{equation}\label{eq:entropy-basic-bounds}
	tI_y(t)\leq H,\qquad \int_{X}d{(y,z)}^2k_t^y(z)\,dm(z)\leq Ht, \qquad |E_y(t)+\log V_t(y)|\leq H.
\end{equation}
For the Fisher estimate, square the gradient bound and divide by the lower heat-kernel bound. The resulting Gaussian has summable annular tails. The entropy estimate follows from
\[ |\log(k_t^y(z)V_t(y))| \leq C+ct+C\frac{d{(y,z)}^2}{t}. \]

\subsection{Rank-sensitive dissipation}
For $f\in\Test(X)$, define the measure $\Gamma(f)$ by
\[ \int_X\phi\,d\Gamma(f) =-\frac12\int_X\langle\nabla \phi,\nabla |\nabla f|^2\rangle\,dm -\int_X\phi\,\langle\nabla f,\nabla \Delta f\rangle\,dm \]
for every $\phi\in\Lip_c(X)$.

\begin{prop}[Improved Bochner inequality {\cite{Han18}}]\label{prop:rank-bochner}
	For every $f\in\Test(X)$, the Bochner measure $\Gamma(f)$ satisfies
	\begin{equation}\label{eq:rank-improved-bochner}
		\Gamma(f) \geq \left[ |\Hess f|^2 + \frac{{(\tr\Hess f-\Delta f)}^2}{N-k} - K|\nabla f|^2 \right]m.
	\end{equation}
\end{prop}

\Needspace{8\baselineskip}
\begin{samepage}
\begin{prop}\label{prop:heat-identities}
	For every $y\in X$, the Fisher information $I_y$ is continuous and locally of bounded variation on $(0,\infty)$, $A_y\in L^1_{\mathrm{loc}}((0,\infty))$, and $dI_y\leq(-2A_y+2KI_y)\,dt$. In particular, for every $0<a<b$,
	\begin{equation}\label{eq:fisher-dissipation}
		I_y(b)+2\int_a^{b}A_y(t)\,dt \leq I_y(a)+2K\int_a^{b}I_y(t)\,dt.
	\end{equation}
\end{prop}
\end{samepage}
\begin{proof}
	 By ~\cite{BGHZ21, BS18},  $k_t^y$ locally bounded away from zero and locally a test function. The Sobolev chain rules therefore apply to $u_t^y$. On compact positive-time intervals, the kernel is continuous with values in $W^{1,2}$, while $\Delta k_t^y\in W^{1,2}$ and $\Hess k_t^y\in L^2$.

	Fix $\epsilon>0$ and set
	\[ J_\epsilon(t)=\int_X\frac{|\nabla k_t^y|^2}{k_t^y+\epsilon}\,dm. \]
	Since the scalar function $r\mapsto-\log(1+r/\epsilon)$ vanishes at zero, its composition with the heat kernel belongs to the global positive-time test class. Apply~\eqref{eq:rank-improved-bochner} to $-\log(1+k_t^y/\epsilon)$ with weight $(k_t^y+\epsilon)\zeta_R$, where $\zeta_R$ is supported in $B_{2R}(y)$, equals one on $B_R(y)$, and satisfies $|\nabla\zeta_R|\leq C/R$. The positive-time derivative estimates bound the time-integrated cutoff terms on $[a,b]$ by $C_{\epsilon,a,b}/R$. Differentiating $J_\epsilon$ and exhausting $X$ consequently gives, for almost every $t>0$,
	\[ J_\epsilon'(t)=-2\int_X(k_t^y+\epsilon)\,d\Gamma\bigl(-\log(1+k_t^y/\epsilon)\bigr). \]
	The Bochner inequality, integrated between $a$ and $b$, yields
	\begin{align*}
		J_\epsilon(b)&+2\int_a^b\int_X(k_t^y+\epsilon)\bigl|\Hess[-\log(1+k_t^y/\epsilon)]\bigr|^2\,dm\,dt\\
		&+\frac{2}{N-k}\int_a^b\int_X(k_t^y+\epsilon)\bigl|(\tr\Hess-\Delta)[-\log(1+k_t^y/\epsilon)]\bigr|^2\,dm\,dt\\
		&\leq J_\epsilon(a)+2K\int_a^b J_\epsilon(t)\,dt.
	\end{align*}
	As $\epsilon\to0^+$, monotone convergence gives $J_\epsilon(t) \to I_y(t)$. On compact subsets of $(0,\infty)\times X$, the positive lower bound for the kernel and the Sobolev chain rules give $L^2$ convergence of the gradients, Hessians, and Laplacians of $-\log(1+k_t^y/\epsilon)$ to those of $u_t^y$. Restricting first to compact balls, using lower semicontinuity for the two nonnegative quadratic terms, and then exhausting $X$ proves~\eqref{eq:fisher-dissipation}.
	The same local derivative convergence, combined with the uniform Gaussian tails obtained from~\eqref{eq:entropy-gaussian}, proves continuity of $I_y$ at positive times. Thus $I_y$ has its actual integral values, is locally of bounded variation, and satisfies
	\[ dI_y\leq(-2A_y+2KI_y)\,dt. \]
	Any singular part of $dI_y$ is nonpositive. Equations~\eqref{eq:entropy-basic-bounds} and~\eqref{eq:fisher-dissipation} also give the dyadic estimate
	\[ \int_\tau^{2\tau}A_y(t)\,dt\leq \frac{C}{\tau}. \]

\end{proof}

\subsection{Measured tangent cones and heat asymptotics}\label{sec:almost-splitting}
Let $(X,d,m)$ be a metric measure space. For $x\in X$ and $r>0$, set
\[ a_x(r)=\int_{B_r(x)} \left(1-\frac{d(x,y)}r\right)\,dm(y), \qquad \gamma_j=\frac{j+1}{\omega_j}\quad(j\geq1). \]
We call $(X,r^{-1}d,a_x{(r)}^{-1}m,x)$ the canonically normalized blow-up at $x$ and scale $r$.

\begin{lem}\label{lem:canonical-measured-tangent}
	Let $(X,d,m)$ be an $\RCD(-K,N)$ space, where $K\geq0$ and $1<N<\infty$. Fix a positive integer $j$ and a point $x\in\cR_j(X)$. Then, as $r\to0^+$,
	\[ (X,r^{-1}d,a_x{(r)}^{-1}m,x) \xrightarrow{\pmGH} (\R^j,d_{\mathrm E},\gamma_j\Leb^j,0). \]
	Consequently, for every $\lambda>0$,
	\[ \lim_{r\to0^+} \frac{m(B_{\lambda r}(x))}{m(B_r(x))} =\lambda^j. \]
\end{lem}

\begin{proof}
	By~\cite[Theorem~2.3]{Kit17}, the canonically normalized blow-ups are precompact, and every limit is an $\RCD(0,N)$ space. The hypothesis identifies its metric with $\R^j$, while the splitting argument in~\cite[Remark~4.3]{Kit17} identifies its measure with $\gamma_j\Leb^j$. Hence the full family converges to this unique normalized limit. Since Euclidean spheres have zero Lebesgue measure,
	\[ \frac{m(B_{\lambda r}(x))}{a_x(r)} \longrightarrow \gamma_j\omega_j\lambda^j=(j+1)\lambda^j. \]
	Dividing by the same limit with $\lambda=1$ proves the ball-volume ratio.
\end{proof}

Harmonic almost splitting maps originate in the theory of Cheeger--Colding~\cite{CC96}. We record the following consequence of the RCD existence theorem in an $L^2$ formulation. Here $\fint_B f\,dm=m(B)^{-1}\int_B f\,dm$; the squared Gram error is equivalent to~\cite[Definition~1.3]{BPS21} after changing the small parameter, by the uniform gradient bound.

\begin{prop}[{\cite{BPS21}}]\label{prop:regular-almost-splitting}
	Let $(X,d,m)$ be an $\RCD(-K,N)$ space, where $K\geq0$ and $1<N<\infty$, and let $z\in\cR_j(X)$ for a positive integer $j$. Given $r_i\to0^+$, put
	\[ d_i=r_i^{-1}d,\qquad m_i=\frac{m}{m(B_{r_i}(z))}. \]
	For every fixed $R>1$ and all sufficiently large $i$, there are numbers $\delta_i\to0$ and $(j,\delta_i)$-almost splitting maps
	\[ h_i=(h_i^1,\ldots,h_i^j):B_{5R}^{d_i}(z)\longrightarrow\R^j,\qquad h_i(z)=0, \]
	with the following properties for every $a,b\in\{1,\ldots,j\}$:
	\begin{enumerate}
		\item[(i)] $\Delta^i h_i^a=0$.
		\item[(ii)] $|\nabla^i h_i^a|\leq C_N$.
		\item[(iii)] $\displaystyle\fint_{B_{5R}^{d_i}(z)}|\langle\nabla^i h_i^a,\nabla^i h_i^b\rangle-\delta_{ab}|^2\,dm_i\leq\delta_i$.
		\item[(iv)] $\displaystyle(5R)^2\fint_{B_{5R}^{d_i}(z)}\sum_{a=1}^j|\Hess^i h_i^a|^2\,dm_i\leq\delta_i$.
	\end{enumerate}
	In particular, the corresponding unnormalized Gram and Hessian integrals over $B_{4R}^{d_i}(z)$ tend to zero.
\end{prop}
\begin{proof}
	Lemma~\ref{lem:canonical-measured-tangent} gives convergence of the canonically normalized blow-ups to Euclidean space, and $KR^2r_i^2\to0$. Apply the cited existence theorem at scale $R$ with synthetic parameter $N+2$ and a one-point residual factor; the gradient constant still depends only on $N$. Multiplication of the measure by a constant preserves the normalized splitting conditions. The uniform gradient bound converts the $L^1$ Gram estimate to the displayed $L^2$ estimate; relabel the small parameter and subtract constants to arrange $h_i(z)=0$. Finally, $m_i(B_{5R}^{d_i}(z))\to(5R)^j$, so the two unnormalized integrals tend to zero.
\end{proof}

\begin{lem}\label{lem:regular-heat-asymptotics}
	Let $z\in\cR_j(X)$ for a positive integer $j$. Then:
	\begin{enumerate}
		\item[(i)] As $t\to0^+$, the entropy satisfies
		\begin{equation}\label{eq:entropy-euclidean-growth}
			E_z(t) = \frac j2\log(1/t)+o(\log(1/t)).
		\end{equation}
		\item[(ii)] The Hessian energy satisfies
		\begin{equation}\label{eq:hessian-liminf}
			\liminf_{t\to0^+} t^2\int_X|\Hess u_t^z|^2k_t^z\,dm \geq\frac j4.
		\end{equation}
	\end{enumerate}
\end{lem}

\begin{proof}
	Lemma~\ref{lem:canonical-measured-tangent} gives the canonically normalized Euclidean measured tangent cone and, in particular,
	\[ \frac{m(B_{r/2}(z))}{m(B_r(z))}\longrightarrow2^{-j} \qquad(r\to0^+). \]
	Fix $0<\eta<j$. There is $r_\eta\in(0,1)$ such that
	\[ 2^{-(j+\eta)}\leq\frac{m(B_{r/2}(z))}{m(B_r(z))}\leq2^{-(j-\eta)} \qquad(0<r\leq r_\eta). \]
	For $r_q=2^{-q}r_\eta$, iteration of these inequalities gives, for every integer $q\geq0$,
	\[ m(B_{r_\eta}(z))\,2^{-q(j+\eta)}\leq m(B_{r_q}(z))\leq m(B_{r_\eta}(z))\,2^{-q(j-\eta)}. \]
	To pass from dyadic scales to arbitrary radii, choose $q$ so that $r_{q+1}<r\leq r_q$. Monotonicity of ball volumes then gives
	\[ m(B_{r_\eta}(z))\,2^{-(q+1)(j+\eta)}\leq m(B_r(z))\leq m(B_{r_\eta}(z))\,2^{-q(j-\eta)}. \]
	Taking negative logarithms and using $-\log r=q\log2+O_\eta(1)$ yields
	\[ j-\eta\leq\liminf_{r\to0^+}\frac{\log m(B_r(z))}{\log r}\leq\limsup_{r\to0^+}\frac{\log m(B_r(z))}{\log r}\leq j+\eta. \]
	Letting $\eta\to0^+$ proves $\log m(B_r(z))/\log r\to j$. At $r=\sqrt t$, this becomes
	\[ \log m(B_{\sqrt t}(z))=-\frac j2\log(1/t)+o(\log(1/t)). \]
	Finally, the entropy bound $|E_z(t)+\log m(B_{\sqrt t}(z))|\leq H$ in~\eqref{eq:entropy-basic-bounds} gives
	\[ E_z(t)=-\log m(B_{\sqrt t}(z))+O(1)=\frac j2\log(1/t)+o(\log(1/t)), \]
	which proves~\eqref{eq:entropy-euclidean-growth}.

	To prove~\eqref{eq:hessian-liminf}, fix an arbitrary decreasing sequence $t_i\to0^+$ and rescale by
	\[ d_i=t_i^{-1/2}d,\qquad m_i=\frac{m}{m(B_{\sqrt{t_i}}(z))}. \]
	Fix $R>1$. Proposition~\ref{prop:regular-almost-splitting}, applied with $r_i=\sqrt{t_i}$, gives $(j,\delta_i)$-almost splitting maps $h_i:B_{5R}^{d_i}(z)\to\R^j$, with $\delta_i\to0$ and $h_i(z)=0$, satisfying the four properties stated there.

	Choose a global test cutoff $\chi_i$ equal to one on $B_R(z)$, supported in $B_{2R}(z)$, and satisfying
	\[ |\nabla^i\chi_i|\leq C/R,\qquad |\Delta^i\chi_i|\leq C/R^2. \]
	Set $\bar{h}_i^a=\chi_{i}h_i^a$. Then $|\nabla^i \bar{h}_i^a|\leq C$, $|\bar{h}_i^a|\leq CR$, and $|\Delta^i \bar{h}_i^a|\leq C/R$ on the cutoff annulus. The Hessian product rule, the Hessian estimate for $\chi_i$, the vanishing core Hessian energy, and $m_i(B_{2R}(z))=O(R^j)$ give
	\begin{equation}\label{eq:cutoff-coordinate-hessian}
		\limsup_i \int_X\sum_{a=1}^j|\Hess^i \bar{h}_i^a|^2\,dm_i \leq CR^{j-2}.
	\end{equation}
	The calculation remains valid for $j=1$.

	Let $k_{i,s}^x$ denote the heat-kernel function of the rescaled normalized space $(X,d_i,m_i)$ with pole $x$. Fix the point $z$ and put
	\[ a_i=\sum_{a=1}^j \bigl(|\nabla^i\bar{h}_i^a|^2+\bar{h}_i^a\Delta^i \bar{h}_i^a\bigr). \]
	The functions $a_i$ are uniformly bounded and converge to $j$ in $L^1(B_R,m_i)$. Heat-kernel stability~\cite[Theorem~3.3]{AHT18} and the Gaussian bounds give, for each fixed $\tau>0$,
	\[ \limsup_i\sup_{\tau\leq s\leq1}\left|\int_X a_i k_{i,s}^z\,dm_i-j\right|\leq CT(R), \qquad T(R)\longrightarrow0. \]
	Here $T(R)$ also bounds the heat-kernel mass outside $B_R$, uniformly for $0<s\leq1$. Since the $\bar{h}_i^a$ are compactly supported test functions vanishing at $z$,
	\[ \frac{d}{ds}\int_X\sum_a(\bar{h}_i^a)^2k_{i,s}^z\,dm_i=2\int_X a_i k_{i,s}^z\,dm_i. \]
	The moment integral on the left tends to zero as $s\to0^+$. The contribution of $(0,\tau)$ is $O(\tau)$. Integrating and then letting $i\to\infty$ and $\tau\to0^+$ yields
	\begin{equation}\label{eq:coordinate-heat-moments}
		\begin{aligned}
		\limsup_i\left|\int_X a_i k_{i,1}^z\,dm_i-j\right|&\leq CT(R),\\
		\limsup_i\left|\int_X\sum_a(\bar{h}_i^a)^2k_{i,1}^z\,dm_i-2j\right|&\leq CT(R).
		\end{aligned}
	\end{equation}

	Let
	\[ C_i=\int_X\left|\sum_{a=1}^j\nabla^i\bar{h}_i^a\otimes\nabla^i\bar{h}_i^a\right|^2k_{i,1}^z\,dm_i. \]
	Weighted integration by parts gives
	\begin{align*}
	&-\int_X\sum_{a=1}^j\bar{h}_i^a\langle\nabla^i \log k_{i,1}^z,\nabla^i\bar{h}_i^a\rangle k_{i,1}^z\,dm_i\\
	&\qquad=\int_X a_i k_{i,1}^z\,dm_i.
	\end{align*}
	Cauchy--Schwarz against $\sum_a(\bar{h}_i^a)^2k_{i,1}^z$ shows that
	\begin{align*}
	&\int_X\sum_{a=1}^j\langle\nabla^i \log k_{i,1}^z,\nabla^i\bar{h}_i^a\rangle^2k_{i,1}^z\,dm_i\\
	&\qquad\geq\frac{\left(\int_X a_i k_{i,1}^z\,dm_i\right)^2}{\int_X\sum_a(\bar{h}_i^a)^2k_{i,1}^z\,dm_i}.
	\end{align*}
	Thus~\eqref{eq:coordinate-heat-moments} gives
	\[ \liminf_i\int_X\sum_{a=1}^j\langle\nabla^i \log k_{i,1}^z,\nabla^i\bar{h}_i^a\rangle^2k_{i,1}^z\,dm_i\geq\frac j2-e(R), \qquad e(R)\longrightarrow0. \]
	The $L^2$ Gram convergence on the core and the same Gaussian tail estimate give
	\[ \limsup_i|C_i-j|\leq CT(R). \]

	The Hessian polarization identity for the compactly supported $\bar{h}_i^a$ gives
	\begin{align*}
	&-\int_X\sum_{a=1}^j\Hess^i \log k_{i,1}^z(\nabla^i\bar{h}_i^a,\nabla^i\bar{h}_i^a)\,k_{i,1}^z\,dm_i\\
	&\qquad=\int_X\sum_{a=1}^j\langle\nabla^i \log k_{i,1}^z,\nabla^i\bar{h}_i^a\rangle^2k_{i,1}^z\,dm_i-R_i,
	\end{align*}
	where
	\begin{align*}
	R_i={}&-\int_X\sum_{a=1}^j\Bigl[\Hess^i\bar{h}_i^a(\nabla^i \log k_{i,1}^z,\nabla^i\bar{h}_i^a)\\
	&\qquad+\langle\nabla^i \log k_{i,1}^z,\nabla^i\bar{h}_i^a\rangle\Delta^i\bar{h}_i^a\Bigr]k_{i,1}^z\,dm_i.
	\end{align*}
	Rewrite every factor $k_{i,1}^z\nabla^i \log k_{i,1}^z$ as $\nabla^i k_{i,1}^z$. On the core, $\Delta^i\bar{h}_i^a=0$, the gradients of $k_{i,1}^z$ are uniformly bounded, and the coordinate Hessian energies tend to zero. Hence the core part of $R_i$ vanishes. On the cutoff annulus, the Gaussian gradient bound is $O(e^{-R^2/5})$. Equation~\eqref{eq:cutoff-coordinate-hessian} and Cauchy--Schwarz bound the Hessian contribution by $CR^{j-1}e^{-R^2/5}$; the bounds $|\Delta^i\bar{h}_i^a|\leq C/R$ and $m_i(B_{2R})=O(R^j)$ give the same estimate for the Laplacian contribution. Therefore
	\[ \limsup_i|R_i|\leq CR^{j-1}e^{-R^2/5}. \]

	Tensor Cauchy--Schwarz now yields
	\begin{align*}
	&\int_X|\Hess^i \log k_{i,1}^z|^2k_{i,1}^z\,dm_i\\
	&\qquad\geq\frac{\left(\int_X\sum_{a=1}^j\langle\nabla^i \log k_{i,1}^z,\nabla^i\bar{h}_i^a\rangle^2k_{i,1}^z\,dm_i-R_i\right)^2}{C_i}.
	\end{align*}
	The left-hand side equals
	\[ t_i^2\int_X|\Hess u_{t_i}^z|^2k_{t_i}^z\,dm; \]
	the additive measure normalization inside the heat-kernel logarithm has zero Hessian, and orthonormal Hessian components acquire one factor $t_i$ under this metric rescaling. Letting $i\to\infty$ and then $R\to\infty$ gives ${(j/2)}^2/j=j/4$. Since the initial sequence $t_i\to0^+$ was arbitrary, this proves~\eqref{eq:hessian-liminf}.
\end{proof}

\subsection{Sublogarithmic trace-defect control}
The entropy asymptotic bounds the total logarithmic dissipation. The Hessian lower bound accounts for its entire leading coefficient, leaving a lower-order contribution from the nonnegative defect term.

\begin{prop}\label{prop:trace-defect}
	Let $z\in\cR_j(X)$ for a positive integer $j$, and fix $b>0$. As $a\to0^+$,
	\begin{equation}\label{eq:trace-defect-square}
		\int_a^b t\int_{X}{(D_t^z)}^2k_t^z\,dm\,dt = o\bigl(\log(b/a)\bigr).
	\end{equation}
	The displayed integral is finite for $0<a<b$.
\end{prop}
\begin{proof}
	For each $y\in X$, the entropy is locally absolutely continuous on $(0,\infty)$ and satisfies, for every $0<a<b$,
	\begin{equation}\label{eq:entropy-chain}
		E_y(b)-E_y(a)=-\int_a^b I_y(t)\,dt.
	\end{equation}
	We justify this chain rule also when $m(X)=\infty$. Set
	\[ h_\epsilon(s)=(s+\epsilon)\log(1+s/\epsilon)-s. \]
	Then $h_\epsilon(0)=h_\epsilon'(0)=0$ and $h_\epsilon''(s)=(s+\epsilon)^{-1}$. Since $k_t^y\in L^2$ at positive times, $h_\epsilon(k_t^y)$ is integrable, and the global $L^2$ chain rule gives
	\[ \left(\int_X h_\epsilon(k_t^y)\,dm\right)'=-\int_X\frac{|\nabla k_t^y|^2}{k_t^y+\epsilon}\,dm. \]
	Using $\int_X k_t^y\,dm=1$, rewrite the regularized entropy as
	\begin{align*}
		\int_X h_\epsilon(k_t^y)\,dm+\log\epsilon+1={}&\int_X k_t^y\log(k_t^y+\epsilon)\,dm\\
		&+\epsilon\int_X\log(1+k_t^y/\epsilon)\,dm.
	\end{align*}
	The first term tends to $E_y(t)$ by the absolute entropy integrability from~\eqref{eq:entropy-gaussian}. The second tends to zero by dominated convergence, since $0\leq\epsilon\log(1+k_t^y/\epsilon)\leq k_t^y$. The scalar constants cancel between the two time endpoints. The regularized Fisher integrals increase to $I_y(t)$ and are bounded by $H/t$ on $[a,b]$, so passage to the limit proves~\eqref{eq:entropy-chain} and local absolute continuity.

	Integrate the measure inequality associated with~\eqref{eq:fisher-dissipation} against $t$ on $[a,b]$. Integration by parts is valid because $I_z$ is continuous and of bounded variation; its possible singular derivative is nonpositive and remains on the favorable side. Using~\eqref{eq:entropy-chain} gives
	\[ 2\int_a^{b}tA_z(t)\,dt \leq aI_z(a)-bI_z(b) +E_z(a)-E_z(b) +2K\int_a^{b}tI_z(t)\,dt. \]
	Equations~\eqref{eq:entropy-basic-bounds} and~\eqref{eq:entropy-euclidean-growth} make the right-hand side
	\[ \frac j2\log(b/a)+o(\log(b/a)). \]
	For every $\eta>0$, equation~\eqref{eq:hessian-liminf} gives a $t_\eta>0$ such that
	\[ t^2\int_X|\Hess u_t^z|^2k_t^z\,dm\geq\frac j4-\eta \qquad(0<t<t_\eta). \]
	Consequently,
	\[ \int_a^b t\int_X|\Hess u_t^z|^2k_t^z\,dm\,dt \geq \frac j4\log(b/a)+o(\log(b/a)). \]
	Since
	\[ A_z(t) = \int_X|\Hess u_t^z|^2k_t^z\,dm + \frac1{N-k}\int_X{(D_t^z)}^2k_t^z\,dm, \]
	the upper and lower logarithmic coefficients coincide, leaving
	\[ \int_a^b t\int_X{(D_t^z)}^2k_t^z\,dm\,dt = o(\log(b/a)). \]
	This proves~\eqref{eq:trace-defect-square}.
\end{proof}

\section{Corrected entropy and lower-dimensional exclusion}\label{sec:entropy-exclusion}
We now prove Theorem~\ref{thm:ricci-regular-dimension} by excluding strongly regular points below the essential dimension. The density input is supplied by Theorem~\ref{thm:log-density-identity}, whose proof is deferred to Sections~\ref{sec:green-asymptotics} and~\ref{sec:stress-limit}. The trace-defect estimate gives the corrected entropy a positive logarithmic coefficient at a proposed lower-dimensional regular point, while comparison with nearby ordinary points gives a conflicting upper bound.

For the estimates below, let $(X,d,m)$ be a complete connected $\RCD(-K,N)$ space with full support, measure finite on bounded sets, and essential dimension $k$, where $K\geq0$ and $2\leq k<N<\infty$. Fix $z\in X$, $R>0$, and $U=B_R(z)$. Let $\Phi=-\log\theta_k$ on the full-measure ordinary regular set and assign arbitrary finite values on its complement. Theorem~\ref{thm:log-density-identity} gives $\Phi\in L^2(U,m)$ and the distributional identity~\eqref{eq:log-density-identity}.

\subsection{Localizing the density correction}
Choose $r_0>0$ with $4r_0<R$ and a Lipschitz cutoff $\chi$ that equals one on $B_{2r_0}(z)$ and vanishes outside $B_{3r_0}(z)$. Put $\Psi=\chi\Phi$ on $U$ and extend it by zero. Then $\Psi\in L^2\cap L^1$ and has compact support. Expanding~\eqref{eq:log-density-identity} with the test factor $\chi\phi$ gives
\begin{equation}\label{eq:localized-density-identity}
	-\int_X\Psi\bigl[\phi\Delta f+\langle\nabla \phi,\nabla f\rangle\bigr]\,dm = \int_X\phi \left[ \chi\bigl(\tr\Hess f-\Delta f\bigr) + \Phi\langle\nabla \chi,\nabla f\rangle \right]\,dm.
\end{equation}

For $y\in X$, set $C_y(t)=-P_t\Psi(y)$.

\begin{lem}\label{lem:correction-derivative}
	With the localization $\Psi=\chi\Phi$ above, there are $t_0,C,c'>0$, independent of $y\in B_{r_0}(z)$, such that $C_y$ is locally absolutely continuous on $(0,t_0)$ and, for almost every $t\in(0,t_0)$,
	\begin{equation}\label{eq:correction-derivative}
		C_y'(t) = -\int_{X}k_t^yD_t^y\,dm+\operatorname{err}_y(t),
	\end{equation}
	where the error satisfies, for every $0<t<t_0$,
	\begin{equation}\label{eq:localized-error}
		\int_0^t|\operatorname{err}_y(s)|\,ds \leq Ce^{-c'/t}.
	\end{equation}
\end{lem}
\begin{proof}
	We justify the use of the heat logarithm in~\eqref{eq:localized-density-identity}. Fix $t>0$ and a compact test cutoff $\eta$ equal to one on a neighborhood of $\supp\chi$. The positive continuous kernel $k_t^y$ has a positive minimum and a finite maximum on this neighborhood. Choose a smooth scalar function $\ell_t$ that agrees with $-\log$ on that range and vanishes near zero. Positive-time heat regularity and the test algebra show that
	\[ f_t=\eta\,\ell_t(k_t^y) \]
	is a global test function and agrees with $u_t^y$ near $\supp\chi$. Likewise, $\eta k_t^y$ is an admissible compactly supported Lipschitz test factor. Locality therefore permits $f=f_t$ and $\phi=\eta k_t^y$ in~\eqref{eq:localized-density-identity}.

	Since
	\[ k_t^y\Delta u_t^y+\langle\nabla k_t^y,\nabla u_t^y\rangle=-\Delta k_t^y, \]
	and $C_y'(t)=-\int_X\Psi\Delta k_t^y\,dm$, the localized identity gives~\eqref{eq:correction-derivative}, with
	\[ \operatorname{err}_y(t) = \int_{X}k_t^y \left[ (1-\chi)D_t^y - \Phi\langle\nabla \chi,\nabla u_t^y\rangle \right]\,dm. \]
	All full expectations are finite for almost every time by~\eqref{eq:fisher-dissipation} and the definition of $A_y$; before the decomposition, the identity consists only of compact finite integrals.

	For $y\in B_{r_0}(z)$, both error regions remain a fixed positive distance from the point. Bishop--Gromov comparison with one fixed ball gives $V_t(y)\geq ct^{N/2}$ uniformly for these points and small $t$. Gaussian annular summation then gives
	\[ \int_{\{1-\chi\neq0\}}k_t^y\,dm\leq Ce^{-c/t}. \]
	Cauchy--Schwarz and the definition of $A_y$ consequently bound the first error term by
	\[ \int_{X}k_t^y(1-\chi)|D_t^y|\,dm \leq Ce^{-c/t}A_y{(t)}^{1/2}, \]
	after absorbing polynomial powers of $t$ into the exponential. On $\supp\nabla\chi$, the same Gaussian bound and $\Phi\in L^2(U)$ give
	\[ {\left( \int_{\supp\nabla\chi}k_t^y\Phi^2\,dm \right)}^{1/2} \leq Ce^{-c/t}\|\Phi\|_{L^2(U)}. \]
	Since $|\langle\nabla \chi,\nabla u_t^y\rangle|\leq\Lip(\chi)|\nabla u_t^y|$, the second error is at most $Ce^{-c/t}I_y{(t)}^{1/2}$. Thus
	\[ |\operatorname{err}_y(t)| \leq Ce^{-c/t}\bigl(A_y{(t)}^{1/2}+I_y{(t)}^{1/2}\bigr). \]
	On each dyadic interval $[\tau,2\tau]$, Cauchy--Schwarz, $\int_\tau^{2\tau}A_y\leq C/\tau$, and $I_y(s)\leq H/s$ show that the integral of the right-hand side is bounded by $Ce^{-c/(2\tau)}$. Summing the dyadic intervals contained in $(0,t)$ proves~\eqref{eq:localized-error}.

\end{proof}

\subsection{Short-time growth of the corrected entropy}
Now we define
\[ \mathcal F_y(t)=tI_y(t)-E_y(t)-\frac k2\log t-2C_y(t). \]

\begin{prop}\label{prop:corrected-entropy}
	Let $z\in\cR_j(X)$ for a positive integer $j$. Then, as $t\to0^+$,
	\begin{equation}\label{eq:exceptional-entropy-growth}
		\mathcal F_z(t) = \frac{k-j}{2}\log(1/t)+o(\log(1/t)).
	\end{equation}
\end{prop}
\begin{proof}
	At the strongly $j$-regular point $z$, Cauchy--Schwarz, the identity $\int_X k_r^z\,dm=1$, and~\eqref{eq:trace-defect-square} give, for every fixed $b>0$,
	\[ \int_t^b\int_X|D_r^z|k_r^z\,dm\,dr \leq \left(\int_t^b r\int_X{(D_r^z)}^2k_r^z\,dm\,dr\right)^{1/2}\!\bigl(\log(b/t)\bigr)^{1/2} = o\bigl(\log(b/t)\bigr). \]
	Integrating~\eqref{eq:correction-derivative} and using~\eqref{eq:localized-error} therefore gives
	\[ C_z(t)=o(\log(1/t)). \]
	Together with~\eqref{eq:entropy-euclidean-growth} and the Fisher bound, this proves~\eqref{eq:exceptional-entropy-growth}.
\end{proof}

\subsection{Comparison with ordinary points}
We first establish a time-propagation bound that is uniform for points near $z$.

\begin{prop}\label{prop:corrected-entropy-inequality}
	There is $t_0>0$ such that, for every $y\in B_{r_0}(z)$ and $0<a<t<t_0$,
	\begin{equation}\label{eq:corrected-entropy-upper}
		\mathcal F_y(t) \leq \mathcal F_y(a)+2KH(t-a) +2\int_a^t|\operatorname{err}_y(s)|\,ds,
	\end{equation}
	where $H$ is the Fisher bound in~\eqref{eq:entropy-basic-bounds} and $\operatorname{err}_y$ is the error from Lemma~\ref{lem:correction-derivative}.
\end{prop}
\begin{proof}
	For each fixed $y\in X$, Proposition~\ref{prop:heat-identities} gives $A_y(t)<\infty$ for almost every $t>0$. Its definition and Cauchy--Schwarz show that $\tr(\Hess u_t^y)$, $D_t^y$, and $\Delta u_t^y$ are integrable against $k_t^y\,dm$. Choose a distance cutoff $\zeta_R$ equal to one on $B_R(y)$, supported in $B_{2R}(y)$, with $|\nabla\zeta_R|\leq C/R$. Test the local Laplacian identity with $\zeta_R k_t^y$ and use $\nabla k_t^y=-k_t^y\nabla u_t^y$. The boundary error is at most $CI_y(t)^{1/2}/R$ and tends to zero. Exhaustion and the definition of $D_t^y$ give, for almost every $t>0$,
	\begin{equation}\label{eq:weighted-traces}
		\begin{aligned}
			\int_X k_t^y\,\Delta u_t^y\,dm&=I_y(t),\\
			\int_X k_t^y\,\tr(\Hess u_t^y)\,dm&=I_y(t)+\int_X k_t^yD_t^y\,dm.
		\end{aligned}
	\end{equation}

	Let $g$ be the rank-$k$ tangent metric tensor. The identity
	\[ \left|\Hess u-\frac{g}{2t}\right|^2 = |\Hess u|^2-\frac{\tr\Hess u}{t}+\frac{k}{4t^2}, \]
	combined with~\eqref{eq:fisher-dissipation},~\eqref{eq:entropy-chain},~\eqref{eq:weighted-traces}, and~\eqref{eq:correction-derivative}, gives
	\[ d\mathcal F_y \leq \left\{ -2t\int_X \left[ \left|\Hess u_t^y-\frac{g}{2t}\right|^2 +\frac{{(D_t^y)}^2}{N-k} \right]k_t^y\,dm + 2KtI_y(t) - 2\operatorname{err}_y(t) \right\}dt. \]
	The nonpositive singular part of $dI_y$ retains the correct sign. Using $tI_y(t)\leq H$ and integrating this measure inequality gives~\eqref{eq:corrected-entropy-upper}.
\end{proof}

\begin{lem}\label{lem:nearby-center-comparison}
	There is $C<\infty$ such that, for all sufficiently small $t>0$ and every $y\in B_{\sqrt t}(z)$,
	\begin{equation}\label{eq:source-comparison}
		\mathcal F_z(t) \leq \mathcal F_y(t)+H + Cd(z,y) \left[ t^{-1/2}(1+\log(1/t)) + t^{-(N+1)/2} \right],
	\end{equation}
	where $H$ is the Fisher bound in~\eqref{eq:entropy-basic-bounds}. The constant $C$ depends only on the fixed localization and geometry.
\end{lem}
\begin{proof}
	For $y,y'\in B_{\sqrt t}(z)$, symmetry and the heat-kernel gradient estimate~\cite[Corollaries~1.1--1.2]{JLZ16} in the first spatial variable give
	\[ |k_t^y(\xi)-k_t^{y'}(\xi)| \leq L_t(\xi)d(y,y'), \]
	where, with $d_{\xi}=d(z,\xi)$ and $V_t=m(B_{\sqrt t}(z))$,
	\[ L_t(\xi) = \frac{C}{\sqrt t\,V_t} {\left(1+\frac{d_{\xi}}{\sqrt t}\right)}^N \exp\left(Cd_{\xi}-\frac{d_{\xi}^2}{10t}\right). \]
	To obtain this pointwise statement, apply the gradient estimate in the first spatial variable on $B_{2\sqrt t}(z)$ and use local Sobolev-to-Lipschitz for the continuous kernel representatives along a geodesic joining $y$ and $y'$. Bishop--Gromov comparison of volumes of balls centered at these points and ${(d_{\xi}-2\sqrt t)}_+^2\geq d_{\xi}^2/2-4t$ give the displayed envelope. On the $q$-th annulus, division by $V_t$ leaves a polynomial in $q$ multiplied by $e^{-q^2/10}$, so annular summation yields
	\[ \int_{X}L_t(\xi)\left(1+\frac{d_{\xi}^2}{t}\right)\,dm(\xi) \leq\frac{C}{\sqrt t}, \qquad \sup_{\xi}L_t(\xi)\leq\frac{C}{\sqrt t\,V_t}. \]
	The Gaussian two-sided bounds also give
	\[ |\log k_t^y(\xi)| \leq C+|\log V_t|+C\frac{d_{\xi}^2}{t}. \]
	Applying the mean-value formula to $s\log s$, whose derivative is $1+\log s$, and then using the preceding two estimates proves
	\[ |E_y(t)-E_{y'}(t)| \leq Cd(y,y')t^{-1/2}\bigl(1+|\log V_t|\bigr). \]
	Similarly, $\Psi\in L^1$ and the supremum bound for $L_t$ give
	\[ |C_y(t)-C_{y'}(t)| \leq \frac{Cd(y,y')}{\sqrt t\,V_t}\|\Psi\|_1. \]
	Taking $y'=z$ and using $V_t\geq ct^{N/2}$, $|\log V_t|\leq C+C\log(1/t)$, and the Fisher bound $0\leq tI_y(t),tI_z(t)\leq H$ proves~\eqref{eq:source-comparison}.
\end{proof}

\subsection{Proof of the main theorem}\label{sec:main-theorem-proof}
We combine the dimension results for a renormalized Ricci-limit measure with the corrected entropy estimates to prove Theorem~\ref{thm:ricci-regular-dimension} directly.

\begin{proof}[Proof of Theorem~\ref{thm:ricci-regular-dimension}]
	If $\cR_1(X)\neq\varnothing$, Chen gives $\cR(X)=\cR_1(X)$~\cite{Chen16}. Hence, we only consider $k\geq2$. $(X,d,m)$ is a complete connected $\RCD(-(n-1),n)$ space. By Colding--Naber~\cite{CN12}, there is a unique essential dimension $k$ such that
	\[ m(X\setminus\cR_k(X))=0. \]

	The result of Kapovitch--Li~\cite{KL18} gives
	\[ \cR_j(X)=\varnothing \qquad(k<j\leq n). \]
	Hence it remains to prove $\cR_j(X)=\varnothing$ for $1\leq j<k$.

	Arguing by contradiction, assume that there exists $z\in\cR_j(X)$ for some $1\leq j<k$. Put $K=n-1$ and choose a finite $N\geq n$ with $N>k$. Fix $R>0$ and choose the localization $\Psi=\chi\Phi$ near $z$ as above. Proposition~\ref{prop:corrected-entropy} gives the first estimate below, and we will prove the second, as $t\to0^+$:
	\[ \mathcal F_z(t)=\frac{k-j}{2}\log(1/t)+o(\log(1/t)), \qquad \mathcal F_z(t)\leq o(\log(1/t)). \]

	We first find ordinary points near $z$ at which the logarithmic density has a sublogarithmic upper bound. Let
	\[ M(r)=\essinf_{B_r(z)}\Phi, \qquad M_+(r)=\max\{M(r),0\}. \]
	Choose $L>5N$ and set
	\[ t=\frac{r^2}{L\log(1/r)}. \]
	The Gaussian upper bound, $V_t(z)\geq ct^{N/2}$, and Cauchy--Schwarz give
	\[ \int_{X\setminus B_r(z)}|\Psi|k_t^z\,dm \leq C\|\Psi\|_2t^{-N/4} \exp\left(-\frac{r^2}{10t}+\frac{ct}{2}\right) \longrightarrow0. \]
	Indeed, after substituting for $t$, the power of $r$ is $L/10-N/2>0$, up to a logarithmic factor. Gaussian annular summation also shows that the heat mass of $B_r(z)$ tends to one. On this ball $\Psi=\Phi$. If $M(r)>0$, then
	\[ P_t\Psi(z) \geq M(r)\int_{B_r(z)}k_t^z\,dm-o(1); \]
	if $M(r)\leq0$, no estimate is required. Since $P_t\Psi(z)=-C_z(t)=o(\log(1/t))$, and $\log(1/t)=2\log(1/r)+O(\log\log(1/r))$, we obtain
	\begin{equation}\label{eq:essential-infimum-control}
		M_+(r)=o(\log(1/r)).
	\end{equation}

	Strong $L^2$ heat continuity supplies a decreasing sequence $a_i\to0^+$ for which
	\[ P_{a_i}\Psi(y)\longrightarrow\Psi(y) \]
	on a common full-measure set $E$; one may choose a subsequence with summable squared $L^2$ errors, which remains valid when $m(X)=\infty$. Intersect $E$ with the ordinary regular set. For $y\in E\cap B_{r_0}(z)$, the density limit gives
	\[ \log\frac{m(B_{\sqrt{a_i}}(y))}{a_i^{k/2}} \longrightarrow \log\omega_k-\Phi(y). \]
	Moreover, $a_{i}I_y(a_i)\leq H$, $-E_y(a_i)\leq\log m(B_{\sqrt{a_i}}(y))+H$, and $-2C_y(a_i)\to2\Phi(y)$. Substitution into the definition of $\mathcal F_y$ gives
	\[ \limsup_i\mathcal F_y(a_i) \leq B_0+\Phi(y), \]
	where $B_0$ is independent of $y$. Letting $a=a_i$ tend to zero in~\eqref{eq:corrected-entropy-upper} and using~\eqref{eq:localized-error} yields
	\begin{equation}\label{eq:ordinary-source-upper}
		\mathcal F_y(t) \leq B_0+\Phi(y)+2KHt+Ce^{-c'/t}.
	\end{equation}

	Choose $M_0>(N+1)/2$. By~\eqref{eq:essential-infimum-control} and the definition of essential infimum, a positive-measure subset of $B_{t^{M_0}}(z)$ satisfies
	\[ \Phi\leq1+M_+(t^{M_0}). \]
	If the essential infimum is $-\infty$, a finite negative sublevel of positive measure has the same role. This subset meets $E$, so choose $y_t\in E\cap B_{t^{M_0}}(z)$. Then
	\[ d(z,y_t)<t^{M_0}, \qquad \Phi(y_t)\leq1+M_+(t^{M_0})=o(\log(1/t)). \]
	For small $t$, the point $y_t$ belongs to both $B_{\sqrt t}(z)$ and $B_{r_0}(z)$. By~\eqref{eq:ordinary-source-upper} and Lemma~\ref{lem:nearby-center-comparison} applied to $y_t$, the comparison errors tend to zero because $M_0>(N+1)/2$. Hence
	\[ \mathcal F_z(t)\leq o(\log(1/t)). \]
	This contradicts the positive coefficient $(k-j)/2$ in~\eqref{eq:exceptional-entropy-growth}. Thus
	\[ \cR_j(X)=\varnothing \qquad(1\leq j<k). \]
	Combining the two exclusions gives
	\[ \cR(X)=\bigsqcup_{j=1}^n\cR_j(X)=\cR_k(X). \]
	The stratum $\cR_k(X)$ is nonempty because it has full measure and $m$ is nonzero.
\end{proof}

\section[The density theorem, Part I: local integrability]{Proof of Theorem~\ref{thm:log-density-identity}, Part I: local integrability}\label{sec:green-asymptotics}
This is the first block of the proof of Theorem~\ref{thm:log-density-identity}. Its goal is part~(i): $\Phi=-\log\theta_k\in L^2_{\mathrm{loc}}(m)$. We first recall the exact Green coefficient, then construct $\Phi_\epsilon$ and prove its convergence to $\Phi$. A uniform directional estimate and local transposition give the local $L^2$ bound. The directional estimate also supplies the compactness used in Part~II to construct the weak directional derivatives $b_f$. The product with $\R^3$ supplies the dimension needed for the singular Green kernel, and the killing factor controls the large-time integral.

\Needspace{24\baselineskip}
\subsection{Green asymptotics and soft Green volumes}\label{sec:soft-green-volumes}
\begin{thm}[Green singularity {\cite{BDS22}}]\label{thm:green-asymptotic}
	Let $(Y,d,m)$ be the metric measure product of a complete $\RCD(-K,N-3)$ space and $(\R^3,d_{\mathrm E},\Leb^3)$, where $K\geq0$ and $4<N<\infty$, and let $k_t^{Y,x}$ be its reference heat-kernel function with pole $x$. Choose $c>c_0(K,N)$, where $c_0$ is an exponential coefficient for which the Gaussian kernel and gradient bounds hold, and define
	\[ G_c(x,y)=\int_0^\infty e^{-ct}k_t^{Y,x}(y)\,dt, \qquad x\neq y. \]
	Fix an integer $\ell$ with $2<\ell\leq N$ and a point $x\in Y$ with positive finite density $\theta_\ell(x):=\lim_{r\to0^+}m(B_r(x))/(\omega_\ell r^\ell)$. Suppose that
	\[ (Y,r^{-1}d,m/m(B_r(x)),x) \xrightarrow{\pmGH} (\R^\ell,d_{\mathrm E},\omega_\ell^{-1}\Leb^\ell,0) \qquad(r\to0^+). \]
	Then
	\[ \lim_{\substack{y\to x\\y\neq x}} d{(x,y)}^{\ell-2}G_c(x,y) = \frac{1}{\ell(\ell-2)\omega_\ell\theta_\ell(x)}. \]
\end{thm}

\begin{rem}\label{rem:green-asymptotic-coefficient}
The coefficient is obtained by applying the Green-kernel blow-up convergence in~\cite[Proposition~2.3]{BDS22} to the Euclidean tangent cone and using its unit-Dirac Green kernel.
\end{rem}

Let $(X,d,m)$ satisfy the hypotheses of Theorem~\ref{thm:log-density-identity}, and write $\theta=\theta_k$. Set
\[ Y=X\times\R^3,\qquad m_Y=m\times\Leb^3,\qquad \ell=k+3>2. \]
The product is $\RCD(-K,N+3)$ and has essential dimension $\ell$. The positive finite intrinsic density and full-measure normalized Euclidean tangent cone property follow from~\cite{BS18}.

Choose a killed Green kernel
\[ G(a,b)=\int_0^\infty e^{-ct}k_t^{Y,a}(b)\,dt \]
with $c$ larger than the Gaussian time coefficient. Fix a smooth nondecreasing $H:\lbrack 0,\infty\rparen\to[0,1]$ with $H=0$ on $[0,1]$ and $H=1$ on $\lbrack 2,\infty\rparen$. Put
\[ A_{H,\ell} = \int_{\R^\ell}H(c_\ell|z|^{2-\ell})\,d\Leb^\ell(z), \qquad V_\epsilon(a) = \int_{Y}H(\epsilon G(a,b))\,dm_Y(b), \]
and
\begin{equation}\label{eq:local-green-regularization}
	\Phi_\epsilon(x) = \frac{\ell-2}{2} \log \left( \frac{V_\epsilon((x,0))} {A_{H,\ell}\epsilon^{\ell/(\ell-2)}} \right)
\end{equation}
where $c_\ell=1/[\ell(\ell-2)\omega_\ell]$. The constant $A_{H,\ell}$ is finite and positive, and $V_\epsilon$ is locally positive for small $\epsilon$, as shown below.

\begin{lem}\label{lem:green-density-convergence}
	Under the hypotheses of Theorem~\ref{thm:log-density-identity}, for every relatively compact open set $U\subset X$, there is $\epsilon_U>0$ such that $\Phi_\epsilon$ is locally Lipschitz on $U$ for $0<\epsilon<\epsilon_U$. As $\epsilon\to0^+$, $\Phi_\epsilon\to\Phi$ $m$-almost everywhere and locally in measure.
\end{lem}
\begin{proof}

	Fix a compact pole set $Q\subset Y$. Gaussian bounds show that $G(a,b)$ is uniformly bounded for $a\in Q$ and $d_Y(a,b)\geq\delta$. The product estimate
	\[ m_Y(B_s(a))\leq C_{Q}s^3 \]
	and the lower Green comparison~\cite[Proposition~2.21]{BS18} show that $G(a,b)\to\infty$ uniformly as $b\to a$ in a fixed compact neighborhood. Thus, for sufficiently small $\epsilon$, the kernel $H(\epsilon G(a,\cdot))$ is supported in a fixed compact set $P$. Its derivative shell stays away from the diagonal. The Green kernel is jointly locally Lipschitz off the diagonal, and $H(\epsilon G)$ is constant near the diagonal. Its support lies in the fixed compact set $P$, so integration shows that $V_\epsilon$ is positive and locally Lipschitz. On compact pole sets its positive minimum makes the normalized logarithm locally Lipschitz as well. Euclidean translation invariance makes it independent of the $\R^3$ coordinate.

	Fix $x$ with positive finite $k$-density and normalized measured tangent cone $\R^k$, and put $a=(x,z)$. Product-ball slicing gives
	\begin{align*}
		m_Y(B_r(a))&=\int_{|w|<r}m\bigl(B_{\sqrt{r^2-|w|^2}}(x)\bigr)\,dw\\
		&=\theta(x)\omega_k\int_{|w|<r}(r^2-|w|^2)^{k/2}\,dw\,(1+o(1))\\
		&=\theta(x)\omega_\ell r^\ell(1+o(1)).
	\end{align*}
	The density remainder is uniform over slice radii in $(0,r]$; the last equality is Euclidean ball slicing. Thus $\theta_\ell^Y(a)=\theta(x)$, and the normalized measured tangent cone is $\R^\ell$. Theorem~\ref{thm:green-asymptotic} gives
	\[ G(a,b)=c_\ell\theta(x)^{-1}d_Y(a,b)^{2-\ell}(1+o(1)) \qquad(b\to a). \]
	Let $q=\ell/(\ell-2)$. For any $0<\delta<1$, the Green asymptotic gives $\rho_\delta>0$ such that
	\[ (1-\delta)c_\ell\theta(x)^{-1}d_Y(a,b)^{2-\ell}\leq G(a,b)\leq(1+\delta)c_\ell\theta(x)^{-1}d_Y(a,b)^{2-\ell} \]
	whenever $0<d_Y(a,b)<\rho_\delta$. The kernel is bounded outside $B_{\rho_\delta}(a)$, so its sufficiently high superlevel sets lie in this ball. With
	\[ r_\pm(s)=\left(\frac{(1\pm\delta)c_\ell s}{\theta(x)}\right)^{1/(\ell-2)}, \]
	for all sufficiently small $s>0$,
	\[ B_{r_-(s)}(a)\setminus\{a\}\subset\{G(a,\cdot)>1/s\}\subset B_{r_+(s)}(a). \]
	Since the pole is null, ball-volume comparison, first as $s\to0^+$ and then $\delta\to0^+$, gives
	\[ m_Y\{G(a,\cdot)>1/s\}=\omega_\ell c_\ell^q\theta(x)^{1-q}s^q(1+o(1)). \]

	Since $H'\geq0$ is supported in $[1,2]$ and $H(0)=0$, the layer-cake identity is
	\[ H(v)=\int_1^2 H'(u)\mathbf1_{\{u<v\}}\,du\qquad(v\geq0). \]
	Tonelli's theorem consequently gives
	\[ V_\epsilon(a)=\int_1^2 H'(u)\,m_Y\{G(a,\cdot)>u/\epsilon\}\,du. \]
	In the preceding superlevel asymptotic, substitute $s=\epsilon/u$. Because $\epsilon/2\leq s\leq\epsilon$ for $1\leq u\leq2$, its relative remainder tends to zero uniformly in $u$ at this fixed pole. Hence
	\[ V_\epsilon(a)=\omega_\ell c_\ell^q\theta(x)^{1-q}\epsilon^q\left(\int_1^2 H'(u)u^{-q}\,du\right)(1+o(1)). \]
	The same identity in the Euclidean model gives
	\begin{align*}
		A_{H,\ell}&=\int_1^2 H'(u)\,\Leb^\ell\{\xi:c_\ell|\xi|^{2-\ell}>u\}\,du\\
		&=\omega_\ell c_\ell^q\int_1^2 H'(u)u^{-q}\,du.
	\end{align*}
	Since $1-q=-2/(\ell-2)$, these identities prove
	\begin{equation}\label{eq:local-green-volume-limit}
		V_\epsilon((x,z)) = A_{H,\ell}\theta{(x)}^{-2/(\ell-2)} \epsilon^{\ell/(\ell-2)}(1+o(1)).
	\end{equation}
	Hence $\Phi_\epsilon\to\Phi$ almost everywhere and in measure on compact sets.
\end{proof}

\subsection{Local control of the logarithmic density}\label{sec:local-transposition}\label{app:transposition}\label{sec:proof-density-integrability}

\begin{prop}\label{prop:uniform-directional}
	Under the hypotheses of Theorem~\ref{thm:log-density-identity}, for relatively compact open sets $A\Subset B\Subset X$, there are constants $C_{A,B}<\infty$ and $\epsilon_{A,B}>0$ such that, for every $f\in\Test(X)$ with $\supp f\Subset B$ and every $0<\epsilon<\epsilon_{A,B}$,
	\begin{equation}\label{eq:local-regularized-directional}
		\|\langle\nabla \Phi_\epsilon,\nabla f\rangle\|_{L^2(A)} \leq C_{A,B} \left( \|\Hess f\|_2+\|\Delta f\|_2 +\|\nabla f\|_2+\|f\|_2 \right).
	\end{equation}
\end{prop}
\begin{proof}
	Fix $A\Subset B$ and choose $\chi\in C_c^\infty(\R^3)$ equal to one on $B_2(0)$ and zero outside $B_3(0)$. For $f\in\Test(X)$ with $\supp f\Subset B$, put $v(x,z)=\chi(z)f(x)$. The product rules give
	\begin{equation}\label{eq:product-test-bound}
		\|\Hess^Y v\|_2+\|\Delta^Y v\|_2 \leq C_\chi\left(\|\Hess f\|_2+\|\Delta f\|_2+\|\nabla f\|_2+\|f\|_2\right).
	\end{equation}
	The derivatives of $v$ and $G$ in this proof are taken on $Y$. Write
	\[ \mathcal B_G(a,\rho)=\{b\in Y:G(a,b)>1/\rho\}. \]
	By the compact localization used in Lemma~\ref{lem:green-density-convergence}, there are a fixed compact product neighborhood $P$ containing $\overline B\times\overline{B_3(0)}$ and a number $\rho_0>0$ such that $\mathcal B_G(a,\rho)\subset P$ for $a\in\overline A\times\overline{B_1(0)}$ and $0<\rho<\rho_0$. Both choices depend only on the fixed sets and geometry.

	The product $Y$ satisfies Assumption~3 of~\cite{BS18}. Since $\nabla^Y v$ is bounded and compactly supported, with symmetric covariant derivative $\Hess^Y v$ and divergence $\Delta^Y v$,~\cite[Proposition~2.27]{BS18} gives $f_v\geq0$ on $P$ such that
	\begin{equation}\label{eq:green-distortion}
		\left|\langle\nabla^Y v(a),\nabla^Y_a G(a,b)\rangle+\langle\nabla^Y v(b),\nabla^Y_b G(a,b)\rangle\right|\leq G(a,b)\bigl(f_v(a)+f_v(b)\bigr)
	\end{equation}
	for $m_Y\otimes m_Y$-almost every $(a,b)\in P\times P$, and
	\begin{equation}\label{eq:green-distortion-l2}
		\|f_v\|_{L^2(P)}\leq C(P)\bigl(\|\Hess^Y v\|_2+\|\Delta^Y v\|_2\bigr).
	\end{equation}
	Differentiating $V_\epsilon$ and integrating by parts in $b$ gives
	\begin{equation}\label{eq:variation-green-volume}
		\begin{aligned}
			&\langle\nabla^Y V_\epsilon,\nabla^Y v\rangle(a)\\
			&\quad=\int_Y\epsilon H'(\epsilon G(a,b))\bigl[\langle\nabla^Y v(a),\nabla^Y_a G(a,b)\rangle+\langle\nabla^Y v(b),\nabla^Y_b G(a,b)\rangle\bigr]\,dm_Y(b)\\
			&\qquad+\int_YH(\epsilon G(a,b))\Delta^Y v(b)\,dm_Y(b).
		\end{aligned}
	\end{equation}
	For every fixed sufficiently small $\epsilon$, this identity holds for almost every allowed pole: the composed kernel has compact support in $P$ and is constant near the pole, while its derivative shell stays away from the diagonal.

	The local quasi-metric and doubling estimates~\cite[Propositions~2.24 and~2.26]{BS18} give
	\[ m_Y(\mathcal B_G(a,\rho))\leq C\,m_Y(\mathcal B_G(a,\rho/2)),\qquad 0<\rho<\rho_0. \]
	For functions on $P$, extended by zero, define
	\[ M_Gh(a)=\sup_{0<\rho<\rho_0}\frac{1}{m_Y(\mathcal B_G(a,\rho))}\int_{\mathcal B_G(a,\rho)}|h|\,dm_Y. \]
	The comparison argument in~\cite[Lemma~2.17]{BS18}, localized using~\cite[Propositions~2.21 and~2.26]{BS18}, gives $M_Gh\leq CMh$ on $A\times B_1(0)$, where $M$ is the metric Hardy--Littlewood maximal operator. The local bound~\cite[equation~(1.7)]{BS18} therefore yields
	\[ \|M_Gh\|_{L^2(A\times B_1(0))}\leq C\|h\|_{L^2(P)}. \]

	Since $V_\epsilon(a)\geq m_Y(\mathcal B_G(a,\epsilon/2))$ and $H(\epsilon G)$ and $\epsilon G H'(\epsilon G)$ are uniformly bounded and supported in $\mathcal B_G(a,\epsilon)$, dividing~\eqref{eq:variation-green-volume} by $V_\epsilon(a)$ gives
	\[ |\langle\nabla^Y\log V_\epsilon,\nabla^Y v\rangle(a)|\leq C\bigl(f_v(a)+M_Gf_v(a)+M_G|\Delta^Y v|(a)\bigr). \]
	On $A\times B_1(0)$, the definition of $\Phi_\epsilon$ and Euclidean translation invariance give
	\[ \langle\nabla \Phi_\epsilon,\nabla f\rangle=\frac{\ell-2}{2}\langle\nabla^Y\log V_\epsilon,\nabla^Y v\rangle. \]
	Taking $L^2$ norms, applying the maximal bound and~\eqref{eq:green-distortion-l2}, and then using~\eqref{eq:product-test-bound} proves~\eqref{eq:local-regularized-directional}. The neighborhood $P$, the upper bound on $\epsilon$, and all constants are independent of $f$.
\end{proof}

For later use, product-ball slicing gives, for every $a\in Y$,
\begin{equation}\label{eq:product-volume-growth}
	m_Y(B_R(a))\geq\left(\frac Rs\right)^3m_Y(B_s(a)),\qquad 0<s<R.
\end{equation}
Indeed, rescaling the Euclidean variable by $R/s$ multiplies its volume element by $(R/s)^3$ and increases every corresponding $X$-slice radius.

\begin{proof}[Proof of Theorem~\ref{thm:log-density-identity}(i)]
	We prove $\Phi\in L^2_{\mathrm{loc}}(m)$ by obtaining uniform local $L^2$ bounds for the regularizations~\eqref{eq:local-green-regularization}.

	Lemma~\ref{lem:green-density-convergence} gives convergence in measure to the finite function $\Phi$, and Proposition~\ref{prop:uniform-directional} gives~\eqref{eq:local-regularized-directional}. We use these established estimates to control the harmonic remainder in a local transposition argument.

	Fix $x\in X$ and choose $r>0$ such that all balls below have compact closure and
	\[ E=B_{16r}(x)\setminus B_{8r}(x) \]
	contains a nonempty open set. Full support gives $m(E)>0$. For a global test function $f$ compactly supported in $B_{8r}(x)$, let $\bar f=\fint_{B_{16r}(x)}f\,dm$ denote its mean over $B_{16r}(x)$. The local Poincar\'e inequality gives
	\[ \int_{B_{16r}(x)}|f-\bar f|^2\,dm \leq C r^2\int_X|\nabla f|^2\,dm. \]
	Since $f=0$ on $E$,
	\[ |\bar f|^2m(E) \leq \int_E|f-\bar f|^2\,dm. \]
	Combining these two inequalities bounds the mean term in $\|f\|_2^2$ and yields the Dirichlet estimate
	\[ \|f\|_2\leq C_{x,r}\|\nabla f\|_2. \]
	Reference integration by parts then gives
	\[ \|\nabla f\|_2^2 = -\int_{X}f\Delta f\,dm \leq C_{x,r}\|\nabla f\|_2\|\Delta f\|_2, \]
	and hence
	\[ \|\nabla f\|_2\leq C_{x,r}\|\Delta f\|_2, \qquad \|f\|_2\leq C_{x,r}^2\|\Delta f\|_2. \]
	The integrated Hessian estimate~\cite[Theorem~2.3, equation~(2.11)]{BGHZ21} now gives
	\[ \|\Hess f\|_2+\|\nabla f\|_2+\|f\|_2 \leq C_{x,r}\|\Delta f\|_2. \]

	Let $F:\R\to\R$ be smooth with $|F'|\leq1$. At each fixed $\epsilon$, local boundedness of $\Phi_\epsilon$ makes $F(\Phi_\epsilon)$ an admissible local Sobolev multiplier. Insert a compact cutoff equal to one near $\supp f$, apply the chain rule, and then use~\eqref{eq:local-regularized-directional}. This proves
	\begin{align}
		&\left|\int_{B_{8r}(x)}F(\Phi_\epsilon)\Delta f\,dm\right|\notag\\
		&\qquad=\left|\int_{B_{8r}(x)}F'(\Phi_\epsilon)\langle\nabla\Phi_\epsilon,\nabla f\rangle\,dm\right|\leq C_{x,r}\|\Delta f\|_{L^2(B_{8r}(x))}.\label{eq:transposition-estimate}
	\end{align}
	The constant $C_{x,r}$ depends only on the fixed ball and geometry and is independent of $F$ and $\epsilon$.

	Choose a smooth nondecreasing $\alpha:\R\to[0,1]$ that vanishes on $\lparen-\infty,0\rbrack$ and equals one on $\lbrack 1,\infty\rparen$, and put
	\[ F_+(s)=\int_0^s\alpha(t)\,dt. \]
	Then
	\[ F_+\geq0,\qquad 0\leq F_+'\leq1,\qquad F_+(s)\leq s_+,\qquad s_+-1\leq F_+(s). \]
	Apply~\eqref{eq:transposition-estimate} separately to
	\[ v_\epsilon=F_+(\Phi_\epsilon) \quad\text{and}\quad v_\epsilon=F_+(-\Phi_\epsilon). \]
	On the linear subspace of $L^2(B_{8r}(x))$ formed by the Laplacians of compactly supported global test functions, the functional
	\[ \Delta f\longmapsto\int_{B_{8r}(x)}v_\epsilon\Delta f\,dm \]
	has norm at most $C_{x,r}$. Extend it to the closure of that subspace and apply Riesz representation. There is $w_\epsilon\in L^2(B_{8r}(x))$ such that
	\[ \|w_\epsilon\|_{L^2(B_{8r}(x))}\leq C_{x,r}, \qquad \int_{B_{8r}(x)}(v_\epsilon-w_\epsilon)\Delta f\,dm=0 \]
	for every such $f$. By Weyl's lemma~\cite[Theorem~1.3]{PZZ25}, the function
	\[ h_\epsilon=v_\epsilon-w_\epsilon \]
	has a locally Lipschitz weakly harmonic representative in $B_{8r}(x)$.

	We use two standard local estimates for this harmonic remainder. On the fixed balls, the Cheeger form is a regular strongly local Dirichlet form with intrinsic distance $d$. Local doubling and Poincar\'e give the required Sobolev inequality~\cite[Theorem~2.6]{Stu96}. The negative part $h_-$ is a nonnegative weakly subharmonic function. Apply the local subsolution estimate~\cite[Theorem~2.1]{Stu95}, with $p=2$, to its stationary heat subsolution, using a finite covering inside $B_{4r}(x)$. This gives
	\[ \sup_{B_{3r}(x)}h_- \leq C\|h_-\|_{L^2(B_{4r}(x))}. \]
	For $H\geq0$ harmonic on $B_{3r}(x)$, the local parabolic Harnack theorem and its elliptic consequence~\cite[Theorem~3.5 and Proposition~3.2]{Stu96} give
	\[ \sup_{B_r(x)}H\leq C_H\inf_{B_r(x)}H. \]
	Both constants depend only on the fixed balls and their doubling and Poincar\'e data.

	Return to $h_\epsilon$. Since $v_\epsilon\geq0$,
	\[ {(h_\epsilon)}_-\leq|w_\epsilon|. \]
	The negative-part estimate and the uniform $L^2$ bound for $w_\epsilon$ therefore give
	\[ h_\epsilon\geq-C_0 \quad\text{on }B_{3r}(x), \]
	with $C_0$ independent of $\epsilon$. Thus $h_\epsilon+C_0$ is nonnegative and harmonic there.

	It remains to anchor this positive harmonic function. By~\eqref{eq:local-green-volume-limit}, $v_\epsilon$ converges in measure on $B_r(x)$ to the finite function $F_+(\Phi)$ or $F_+(-\Phi)$. Choose $B_0$ so that the set where this limit exceeds $B_0$ has measure less than $m(B_r(x))/8$. For sufficiently small $\epsilon$,
	\[ m\bigl(B_r(x)\cap\{v_\epsilon>B_0+1\}\bigr) <\frac14m(B_r(x)). \]
	Choose $J$ with
	\[ {(C_{x,r})}^2J^{-2}<\frac14m(B_r(x)). \]
	Chebyshev's inequality gives
	\[ m\bigl(B_r(x)\cap\{|w_\epsilon| > J\}\bigr) <\frac14m(B_r(x)). \]
	Hence a positive-measure subset of $B_r(x)$ satisfies both $v_\epsilon\leq B_0+1$ and $|w_\epsilon|\leq J$. At a point of this subset where the continuous harmonic representative agrees with $v_\epsilon-w_\epsilon$,
	\[ h_\epsilon+C_0\leq B_0+1+J+C_0. \]
	The Harnack estimate now bounds the supremum of $h_\epsilon+C_0$ on $B_r(x)$ by the same fixed data. Since $v_\epsilon=h_\epsilon+w_\epsilon$, this supremum bound and $\|w_\epsilon\|_2\leq C_{x,r}$ give a uniform $L^2(B_r(x))$ bound for $v_\epsilon$.

	Apply the argument to both signs. The inequality
	\[ |\Phi_\epsilon| \leq F_+(\Phi_\epsilon)+F_+(-\Phi_\epsilon)+1 \]
	then gives a uniform $L^2(B_r(x))$ bound for $\Phi_\epsilon$. A finite covering gives, for every compact $Q\subset X$, some $\epsilon_Q>0$ such that
	\[ \sup_{0<\epsilon<\epsilon_Q}\|\Phi_\epsilon\|_{L^2(Q,m)}<\infty. \]
	This uniform bound combines the transposition estimate~\eqref{eq:transposition-estimate} with the local harmonic estimates above; convergence in measure anchors the harmonic remainder. Fatou's lemma and the almost-everywhere convergence in Lemma~\ref{lem:green-density-convergence} now give $\Phi\in L^2_{\mathrm{loc}}(m)$, as required in Theorem~\ref{thm:log-density-identity}(i).

	The local $L^2$ bounds and convergence in measure identify every weakly convergent subsequence with $\Phi$. Thus the full family converges weakly in $L^2$ on compact sets. Uniform integrability of $|\Phi_\epsilon-\Phi|^p$ gives strong local $L^p$ convergence for $1\leq p<2$. 
\end{proof}

\section[The density theorem, Part II: the logarithmic-density identity]{Proof of Theorem~\ref{thm:log-density-identity}, Part II: the logarithmic-density identity}\label{sec:stress-limit}
This is the second block of the proof of Theorem~\ref{thm:log-density-identity}. Part~I has established the local $L^2$ integrability of $\Phi$ and the weak local $L^2$ convergence $\Phi_\epsilon\rightharpoonup\Phi$. We first construct the weak directional derivatives $b_f$ and then identify $b_f=\tr(\Hess f)-\Delta f$ on each relatively compact test region. For this identification, we establish the two-endpoint identity for truncated Green kernels, derive the finite-scale stress formula, and evaluate its Euclidean limit. The final subsection inserts the resulting expression for $b_f$ into its distributional representation.

Fix relatively compact open sets $A\Subset B\subset X$ and $f\in\Test(X)$ with $\supp f\Subset B$. By~\eqref{eq:local-regularized-directional}, a subsequence of $\langle\nabla \Phi_\epsilon,\nabla f\rangle$ converges weakly in $L^2(A)$ to some $b_f$. For compactly supported Lipschitz $\phi$,
\[ \int_A\phi\,\langle\nabla \Phi_\epsilon,\nabla f\rangle\,dm = -\int_X \Phi_\epsilon \bigl[\phi\Delta f+\langle\nabla \phi,\nabla f\rangle\bigr]\,dm. \]
The bracket belongs to $L^2$ on a compact set. Passing to the weak limits gives
\[ \int_A\phi b_f\,dm = -\int_X\Phi\bigl[\phi\Delta f+\langle\nabla \phi,\nabla f\rangle\bigr]\,dm \qquad(\phi\in\Lip_c(A)). \]
Density of compactly supported Lipschitz tests in $L^2(A)$ gives uniqueness. Lower semicontinuity and~\eqref{eq:local-regularized-directional} give
\begin{equation}\label{eq:local-density-directional}
	\|b_f\|_{L^2(A)} \leq C_{A,B} \left( \|\Hess f\|_2+\|\Delta f\|_2 +\|\nabla f\|_2+\|f\|_2 \right),
\end{equation}
where $C_{A,B}<\infty$ is independent of $f$.

\subsection{The finite-scale stress formula}\label{sec:truncated-green-stress}
We retain the notation and choices of Section~\ref{sec:soft-green-volumes}. Let $Y=X\times\R^3$, let $G$ be the killed kernel of Section~\ref{sec:green-asymptotics}, and write $A=A_{H,\ell}$. On compact sets,
\[ \Phi_\epsilon\rightharpoonup\Phi \quad\text{in }L^2, \]
and their directional derivatives satisfy the uniform bound~\eqref{eq:local-regularized-directional}.

Fix a compactly supported $f\in\Test(X)$. Choose $\chi\in C_c^\infty(\R^3)$ equal to one on $B_2(0)$ and zero outside $B_3(0)$, and put
\[ v(x,z)=\chi(z)f(x). \]
Derivatives of $v$ and the Green kernels are taken on $Y$. Write $G_a=G(a,\cdot)$ and let $g_Y$ be the tangent metric tensor. Define the shell potential by
\[ W_{\epsilon,a}(z)=\epsilon\int_YG(z,b)H'(\epsilon G(a,b))\,dm_Y(b). \]
Fix a relatively compact product pole set $Q\subset X\times B_1(0)$ and a compact neighborhood $P$ containing $\supp v$ and the Green superlevel sets used below.

\begin{prop}\label{prop:finite-stress}
	For every fixed sufficiently small $\epsilon>0$, the following identity holds for $m_Y$-almost every $a=(x,z)\in Q$:
	\begin{equation}\label{eq:general-finite-stress}
		\begin{aligned}
			&\frac{2V_\epsilon(a)}{\ell-2}\langle\nabla \Phi_\epsilon,\nabla f\rangle(x)\\
			&\quad=\int_Y\bigl(2\Hess^Y v-(\Delta^Y v)g_Y\bigr)(\nabla^Y G_a,\nabla^Y W_{\epsilon,a})\,dm_Y\\
			&\qquad\phantom{={}}-c\int_Y(\Delta^Y v)G_aW_{\epsilon,a}\,dm_Y+\int_YH(\epsilon G_a)\Delta^Y v\,dm_Y.
		\end{aligned}
	\end{equation}
\end{prop}
\begin{proof}
	We use the symmetric-derivative Green-kernel calculation of~\cite{BS18}. For compactly supported $v\in\Test(Y)$, the calculation gives, for $(m_Y\otimes m_Y)$-almost every $(a,b)$,
	\begin{align}
		&\langle\nabla^Y v(a),\nabla^Y_a G(a,b)\rangle+\langle\nabla^Y v(b),\nabla^Y_b G(a,b)\rangle\notag\\
		&\quad=\int_Y(2\Hess^Y v-(\Delta^Y v)g_Y)(\nabla^Y G_a,\nabla^Y G_b)\,dm_Y\notag\\
		&\qquad-c\int_Y(\Delta^Y v)G(a,\cdot)G(b,\cdot)\,dm_Y.
		\label{eq:exact-two-endpoint-stress}
	\end{align}
	To justify the identity and its integrability, define
	\[ G_\delta(a,z)=\int_\delta^\infty e^{-ct}k_t^{Y,a}(z)\,dt. \]
	By~\cite[equation~(2.50)]{BS18},
	\[ (c-\Delta^Y_z)G_\delta(a,z)=e^{-c\delta}k_\delta^{Y,a}(z). \]
	Apply~\cite[equation~(1.18)]{BS18} with vector field $\nabla^Y v$, the two truncated kernels, and a compact test cutoff equal to one near $\supp v$, as in Step~1 of the proof of~\cite[Proposition~2.18]{BS18}. This gives the regularized form of~\eqref{eq:exact-two-endpoint-stress}. The killing term has the displayed negative sign because
	\begin{align*}
		&c\int_Y\langle\nabla^Y v,\nabla^Y(G_\delta(a,\cdot)G_\delta(b,\cdot))\rangle\,dm_Y\\
		&\qquad=-c\int_Y(\Delta^Y v)G_\delta(a,\cdot)G_\delta(b,\cdot)\,dm_Y.
	\end{align*}

	The Gaussian gradient bounds~\cite[Corollaries~1.1--1.2]{JLZ16}, integrated in the pole variable and then in time, give, for $0<\delta<1$,
	\begin{align}
		\sup_{z\in\supp v}\int_Y|\nabla^Y_zG(z,a)|\,dm_Y(a)&\leq C,\label{eq:green-gradient-first-moment}\\
		\sup_{z\in\supp v}\int_Y|\nabla^Y_z(G-G_\delta)(z,a)|\,dm_Y(a)&\leq C\sqrt\delta.\label{eq:green-truncation-first-moment}
	\end{align}
	Here the short-time gradient moment is $O(t^{-1/2})$; for $t\geq1$, the semigroup property bounds it by its time-one value. Mass preservation also gives
	\[ \int_YG(z,a)\,dm_Y(a)=c^{-1},\qquad \int_Y(G-G_\delta)(z,a)\,dm_Y(a)\leq\delta. \]
	Since $2\Hess^Y v-(\Delta^Y v)g_Y$ and $\Delta^Y v$ are compactly supported and in $L^1$, Tonelli gives absolute integrability of both stress terms in the pole pair. The same bounds give local $L^1$ errors $O(\sqrt\delta)$ and $O(\delta)$ for their truncated versions. On the other side, $L^1$ heat contraction and strong continuity give convergence of the two heat averages to the two gradient pairings on the left of~\eqref{eq:exact-two-endpoint-stress}, exactly as the proof of Proposition~2.18 and the proof of Proposition~2.27 in~\cite{BS18}, compare also its equation~(2.51). This proves~\eqref{eq:exact-two-endpoint-stress} in local $L^1(m_Y\otimes m_Y)$.

	For the product test function fixed above, multiply~\eqref{eq:exact-two-endpoint-stress} by $\epsilon H'(\epsilon G(a,b))$ and integrate in $b$. This weight is bounded and supported in $P$, so the established absolute integrability permits Fubini. The definition of $W_{\epsilon,a}$ and weak differentiation of the resolvent identify the resulting integrals with the first two terms on the right of~\eqref{eq:general-finite-stress}. Substitution into~\eqref{eq:variation-green-volume} proves the formula, since on $Q$, where $\chi=1$,
	\[ \langle\nabla^Y V_\epsilon,\nabla^Y v\rangle(a)=\frac{2V_\epsilon(a)}{\ell-2}\langle\nabla\Phi_\epsilon,\nabla f\rangle(x). \]
\end{proof}

\subsection{Shell-potential estimates}
Fix a decreasing sequence $\epsilon_i\to0^+$ and intersect the full-measure validity sets of~\eqref{eq:general-finite-stress}. Choose a pole $a\in Q$ with positive finite $\ell$-density
\[ \theta=\theta_\ell^Y(a), \]
unique normalized Euclidean measured tangent cone, and the coefficient Lebesgue-point properties used below. These conditions hold almost everywhere. Set
\[ r={\left(\frac{\epsilon}{\theta}\right)}^{1/(\ell-2)}, \qquad \epsilon=\theta r^{\ell-2}. \]
For a fixed small $\rho>0$,
\[ c_{a}s^\ell\leq m_Y(B_s(a))\leq C_{a}s^\ell \qquad(0<s<\rho). \]
Theorem~\ref{thm:green-asymptotic} gives
\[ G_a(z) = c_\ell\theta^{-1}d_Y{(a,z)}^{2-\ell}(1+o(1)) \]
along every approach to $a$. The fixed-pole gradient estimate gives
\[ |\nabla^Y G_a|(z)\leq C_{a}d_Y{(a,z)}^{1-\ell}. \]
Hence the forcing term $\epsilon H'(\epsilon G_a)$ is supported in an annulus
\[ c_{a}r<d_Y(a,\cdot)<C_{a}r, \]
and~\eqref{eq:local-green-volume-limit}, together with $\epsilon=\theta r^{\ell-2}$, gives
\[ V_\epsilon(a)=A_{H,\ell}\theta r^\ell(1+o(1)) \qquad(\epsilon\to0^+). \]
In particular, for all sufficiently small $\epsilon>0$,
\[ \frac12 A_{H,\ell}\theta r^\ell\leq V_\epsilon(a)\leq2A_{H,\ell}\theta r^\ell. \]
The two comparison constants are independent of $\epsilon$.

\begin{lem}\label{lem:shell-potential}
	Let $a\in Q$ have positive finite $\ell$-density $\theta$ and unique normalized Euclidean measured tangent cone $\R^\ell$. Set $r=(\epsilon/\theta)^{1/(\ell-2)}$ and $W=W_{\epsilon,a}$. There are $\rho>0$, $C_1\geq1$, $C_a<\infty$, and $\epsilon_a>0$, independent of $\epsilon$, such that, for $0<\epsilon<\epsilon_a$ and $m_Y$-almost every evaluation point $z$:
	\begin{enumerate}
		\item[(i)] On $B_{C_1r}(a)$,
		\begin{equation}\label{eq:W-inner}
			|W(z)|\leq C_a\epsilon r^2, \qquad |\nabla^Y W|(z)\leq C_a\epsilon r.
		\end{equation}
		\item[(ii)] If $s=d_Y(a,z)$ satisfies $C_1r\leq s<\rho$,
		\begin{equation}\label{eq:W-outer}
			|W(z)|\leq C_a\epsilon r^{\ell}s^{2-\ell}, \qquad |\nabla^Y W|(z)\leq C_a\epsilon r^{\ell}s^{1-\ell}.
		\end{equation}
	\end{enumerate}
	On $P\setminus B_\rho(a)$, one also has $|W|+|\nabla^Y W|\leq C_a\epsilon r^\ell$.
\end{lem}
\begin{proof}

	We use the Gaussian kernel and gradient estimates~\cite[Theorems~1.1--1.2 and Corollaries~1.1--1.2]{JLZ16}, together with the product-volume growth~\eqref{eq:product-volume-growth}. The shell-dependent powers follow from the heat representation
	\[ W(z)=\int_0^\infty e^{-ct}P_t^Y\bigl(\epsilon H'(\epsilon G_a)\bigr)(z)\,dt. \]
	The shell description and the density bounds at $a$ give
	\[ \|\epsilon H'(\epsilon G_a)\|_\infty\leq C_H\epsilon, \qquad \|\epsilon H'(\epsilon G_a)\|_1\leq C_a\epsilon\theta r^\ell. \]
	Fix a sufficiently large constant $C_1$. If $z\in B_{C_1r}(a)$, split the heat integral at $r^2$. The $L^\infty$ contraction and its integrated gradient estimate yield
	\[ \int_0^{r^2}e^{-ct}|P_t^Y\bigl(\epsilon H'(\epsilon G_a)\bigr)(z)|\,dt \leq C\epsilon r^2, \]
	and
	\[ \int_0^{r^2}e^{-ct}|\nabla^Y P_t^Y\bigl(\epsilon H'(\epsilon G_a)\bigr)|(z)\,dt \leq C\epsilon\int_0^{r^2}t^{-1/2}\,dt \leq C\epsilon r. \]
	For $t\geq r^2$, the Gaussian kernel and gradient estimates, comparison of balls centered in $B_{C_1r}(a)$, and the shell support of the forcing term give
	\[ |P_t^Y\bigl(\epsilon H'(\epsilon G_a)\bigr)(z)| \leq \frac{C_a\epsilon\,m_Y(B_{C_{a}r}(a))} {m_Y(B_{\sqrt t}(a))}, \]
	\[ |\nabla^Y P_t^Y\bigl(\epsilon H'(\epsilon G_a)\bigr)|(z) \leq \frac{C_a\epsilon\,m_Y(B_{C_{a}r}(a))} {\sqrt t\,m_Y(B_{\sqrt t}(a))}. \]
	By~\eqref{eq:product-volume-growth}, the quotient of ball measures is at most $C_a{(r/\sqrt t)}^3$. Integration from $r^2$ onward, with the killing factor absorbing the large-time curvature exponential, gives respectively $C_a\epsilon r^2$ and $C_a\epsilon r$. This proves~\eqref{eq:W-inner}.

	Now let $s=d_Y(a,z)\geq C_1r$ and $s<\rho$. For every $b\in\supp\bigl(H'(\epsilon G_a)\bigr)$, the choice of $C_1$ gives $d_Y(z,b)\asymp s$. Split the kernel representation of $G(z,b)$ at time $s^2$. For $0<t<s^2$, Bishop--Gromov comparison and $B_{s-C_{a}r}(a)\subset B_s(b)$ give
	\[ m_Y(B_{\sqrt t}(b)) \geq C_a^{-1} {\left(\frac{\sqrt t}{s}\right)}^{N+3} \theta s^\ell. \]
	The Gaussian separation factor therefore gives
	\[ \int_0^{s^2}e^{-ct}k_t^{Y,z}(b)\,dt \leq C_a\theta^{-1}s^{N+3-\ell} \int_0^{s^2}t^{-(N+3)/2}e^{-c_{a}s^2/t}\,dt \leq C_a\theta^{-1}s^{2-\ell}. \]
	Including the factor $t^{-1/2}$ from the heat-kernel gradient bound gives $C_a\theta^{-1}s^{1-\ell}$. For $t\geq s^2$, the product-volume inequality implies
	\[ m_Y(B_{\sqrt t}(b)) \geq C_a^{-1}\theta s^\ell {\left(\frac{\sqrt t}{s}\right)}^3. \]
	The resulting time integrals, again with the killing factor at large time, are bounded by $C_a\theta^{-1}s^{2-\ell}$ for the kernel and $C_a\theta^{-1}s^{1-\ell}$ for its gradient. Thus
	\[ G(z,b)\leq C_a\theta^{-1}s^{2-\ell}, \qquad |\nabla^Y_z G(z,b)|\leq C_a\theta^{-1}s^{1-\ell}. \]
	Integration against the forcing term, whose mass is at most $C_a\epsilon\theta r^\ell$, proves~\eqref{eq:W-outer}.
	On the fixed compact part of $P\setminus B_\rho(a)$, the shell and the evaluation point remain uniformly separated. The corresponding kernel and gradient bounds give
	\[ |W|+|\nabla^Y W|\leq C_a\epsilon r^\ell. \]
\end{proof}

\subsection{The normalized stress limit}
\begin{prop}\label{prop:normalized-stress}
	Fix a decreasing sequence $\epsilon_i\to0^+$. For $m_Y$-almost every $a\in Q$, the following limits hold as $i\to\infty$:
	\begin{enumerate}
		\item[(i)] The normalized stress converges to the trace of the frozen tensor:
		\begin{align*}
			&\frac1{V_{\epsilon_i}(a)}\int_Y\bigl(2\Hess^Y v-(\Delta^Y v)g_Y\bigr)(\nabla^Y G_a,\nabla^Y W_{\epsilon_i,a})\,dm_Y\\
			&\qquad\longrightarrow\frac{2\tr(\Hess^Y v)(a)-\ell\Delta^Y v(a)}{\ell-2}.
		\end{align*}
		\item[(ii)] The normalized killing term vanishes:
		\begin{equation}\label{eq:killing-vanishes}
			\frac1{V_{\epsilon_i}(a)} \int_Y|\Delta^Y v|G_a W_{\epsilon_i,a}\,dm_Y \longrightarrow0.
		\end{equation}
	\end{enumerate}
\end{prop}
\begin{proof}
	We use the local Sobolev stability results of~\cite{AH18}. Work at a pole $a$ with the density and tangent cone properties of Lemma~\ref{lem:shell-potential}, and write $\theta=\theta_\ell^Y(a)$. The Lebesgue-point conditions imposed below define one full-measure set.

	For part~(i), at a Lebesgue point of $|2\Hess^Y v-(\Delta^Y v)g_Y|$, Lemma~\ref{lem:shell-potential}, $|\nabla^Y G_a|\leq C_as^{1-\ell}$, and the lower bound $V_\epsilon(a)\geq\frac12 A_{H,\ell}\theta r^\ell$ give, on dyadic shells of radius $s$,
	\begin{equation}\label{eq:inner-stress-shell}
		\begin{aligned}
			&\frac1{V_\epsilon(a)}\int_{\text{shell}}|2\Hess^Y v-(\Delta^Y v)g_Y|\,|\nabla^Y G_a|\,|\nabla^Y W|\,dm_Y\\
			&\quad\leq C_a\frac sr\fint_{B_{2s}(a)}|2\Hess^Y v-(\Delta^Y v)g_Y|\,dm_Y\qquad(2s<C_1r),
		\end{aligned}
	\end{equation}
	and
	\begin{equation}\label{eq:outer-stress-shell}
		\begin{aligned}
			&\frac1{V_\epsilon(a)}\int_{\text{shell}}|2\Hess^Y v-(\Delta^Y v)g_Y|\,|\nabla^Y G_a|\,|\nabla^Y W|\,dm_Y\\
			&\quad\leq C_a\left(\frac rs\right)^{\ell-2}\fint_{B_{2s}(a)}|2\Hess^Y v-(\Delta^Y v)g_Y|\,dm_Y\qquad(C_1r<s<\rho/2).
		\end{aligned}
	\end{equation}
	The centered averages are bounded. Geometric summation bounds the inner region $B_{\delta r}(a)$ by $C_a\delta$ and the region $B_{\rho/2}(a)\setminus B_{Rr}(a)$ by $C_aR^{2-\ell}$, for small $\delta$ and large $R$. The remaining separated region contributes $O_a(\epsilon)\|2\Hess^Y v-(\Delta^Y v)g_Y\|_1$.

	By the density and generation results in~\cite{Gig18}, followed by localization, choose a countable family of compactly supported Lipschitz functions $h_\alpha\in D(\Delta^Y)$ whose differentials generate the cotangent module on $Q$. Take a common Lebesgue point of their Gram entries, $|\Delta^Y h_\alpha|^2$, all pairings between their tensor products and the stress tensor, the tensor-product Gram entries, and $|2\Hess^Y v-(\Delta^Y v)g_Y|^2$. Choose $\ell$ members spanning the fiber at $a$, and coefficients $t_{\alpha m}$ so that their tensor combination agrees there with $2\Hess^Y v-(\Delta^Y v)g_Y$. Expanding the squared error gives
	\begin{equation}\label{eq:stress-freezing}
		\fint_{B_s(a)}\left|2\Hess^Y v-(\Delta^Y v)g_Y-\sum_{\alpha,m=1}^\ell t_{\alpha m}\,d^Y h_\alpha\otimes d^Y h_m\right|^2\,dm_Y\longrightarrow0.
	\end{equation}
	On $B_{Rr}(a)\setminus B_{\delta r}(a)$, Cauchy--Schwarz and the gradient bounds make this coefficient error tend to zero. The same tail bounds apply to the finite frozen tensor. Thus it suffices to pass that finite sum to the Euclidean tangent cone on fixed annuli.

	Set
	\[ r_i=(\epsilon_i/\theta)^{1/(\ell-2)},\qquad d_i=r_i^{-1}d_Y,\qquad m_i=\frac{m_Y}{\theta r_i^\ell}, \]
	and
	\[ g_i=\epsilon_iG_a,\qquad z_i=\frac{W_{\epsilon_i,a}}{\epsilon_i r_i^2},\qquad h_{\alpha,i}=\frac{h_\alpha-h_\alpha(a)}{r_i}. \]
	The rescaled spaces converge to $(\R^\ell,d_{\mathrm E},\Leb^\ell,0)$. Theorem~\ref{thm:green-asymptotic}, based on~\cite[Proposition~2.3]{BDS22}, gives $g_i\to g_0=c_\ell|z|^{2-\ell}$ locally uniformly off the origin. The rescaled equations are
	\[ \Delta^i g_i=cr_i^2g_i\quad\text{off the pole},\qquad \Delta^i h_{\alpha,i}=r_i\Delta^Y h_\alpha, \]
	and
	\begin{equation}\label{eq:rescaled-W-equation}
		(-\Delta^i+cr_i^2)z_i=H'(g_i).
	\end{equation}
	Lemma~\ref{lem:shell-potential} gives uniform local Lipschitz bounds for $z_i$ and the decay bounds $|z_i|\leq C|z|^{2-\ell}$, $|\nabla^i z_i|\leq C|z|^{1-\ell}$ outside a fixed rescaled ball. The functions $h_{\alpha,i}$ are uniformly Lipschitz and vanish at the pole; their Laplacians tend to zero locally in $L^2$ by the chosen Lebesgue-point property. Also $H'(g_i)\to H'(g_0)$ strongly in local $L^2$, with support in one fixed annulus.

	Local compactness and~\cite[Theorem~4.4]{AH18} now give, on a subsequence, strong local $W^{1,2}$ convergence for $z_i$ and $h_{\alpha,i}$ on balls, and for $g_i$ on annuli. The limiting equation is $-\Delta^{\mathrm E}z_0=H'(g_0)$ on all of $\R^\ell$. The decay bounds identify $z_0$ as the unique decaying Newton potential, hence radial. Each limit $h_{\alpha,0}$ is globally Lipschitz and harmonic, hence affine by the Euclidean Liouville theorem. The Gram limits identify its constant gradients through
	\[ \nabla^{\mathrm E}h_{\alpha,0}\cdot\nabla^{\mathrm E}h_{m,0}=\langle\nabla^Y h_\alpha,\nabla^Y h_m\rangle(a). \]
	They therefore represent the same frozen tensor on the Euclidean tangent cone.

	By polarization using~\cite[Corollary~4.3 and Theorem~4.6]{AH18}, the pairings of $\nabla^i h_{\alpha,i}$ with $\nabla^i g_i$ and $\nabla^i z_i$ converge strongly in $L^1$ on fixed annuli. Their uniform bounds upgrade this to strong $L^2$, so products converge in $L^1$. The tensor and gradient scalings are
	\[ r_i^{-2}\sum_{\alpha,m}t_{\alpha m}\,d^Y h_\alpha\otimes d^Y h_m=\sum_{\alpha,m}t_{\alpha m}\,d^i h_{\alpha,i}\otimes d^i h_{m,i}, \]
	and, in orthonormal components,
	\[ \nabla^Y G_a=(\epsilon_i r_i)^{-1}\nabla^i g_i,\qquad \nabla^Y W_{\epsilon_i,a}=\epsilon_i r_i\nabla^i z_i. \]
	The gradient factors cancel, while $dm_Y=\theta r_i^\ell\,dm_i$ and $V_{\epsilon_i}(a)/(\theta r_i^\ell)\to A$. Passing to the limit on fixed annuli, using~\eqref{eq:stress-freezing}, and then sending $\delta\to 0$ and $R\to \infty$ by~\eqref{eq:inner-stress-shell}--\eqref{eq:outer-stress-shell}, gives
	\begin{equation}\label{eq:stress-tangent-limit}
		\begin{aligned}
			&\lim_i\frac1{V_{\epsilon_i}(a)}\int_Y\bigl(2\Hess^Y v-(\Delta^Y v)g_Y\bigr)(\nabla^Y G_a,\nabla^Y W_{\epsilon_i,a})\,dm_Y\\
			&\quad=\frac1A\int_{\R^\ell}\bigl(2\Hess^Y v-(\Delta^Y v)g_Y\bigr)(a)(\nabla^{\mathrm E}g_0,\nabla^{\mathrm E}z_0)\,d\Leb^\ell.
		\end{aligned}
	\end{equation}

	Homogeneity and the definition of $A$ give
	\[ \int_{\R^\ell}H(\lambda g_0)\,d\Leb^\ell=\lambda^{\ell/(\ell-2)}A\qquad(\lambda>0). \]
	Differentiating at $\lambda=1$ is legitimate because $H'$ is supported away from the pole. Green's identity and rotational invariance therefore give
	\[ \int_{\R^\ell}\nabla^{\mathrm E}g_0\cdot\nabla^{\mathrm E}z_0\,d\Leb^\ell=\int_{\R^\ell}g_0H'(g_0)\,d\Leb^\ell=\frac{\ell A}{\ell-2}, \]
	and
	\begin{equation}\label{eq:euclidean-stress-coefficient}
		\int_{\R^\ell}\nabla^{\mathrm E}g_0\otimes\nabla^{\mathrm E}z_0\,d\Leb^\ell=\frac{A}{\ell-2}\Id_\ell.
	\end{equation}
	Here the annular Green identity has zero inner boundary term because $z_0$ is bounded and radial harmonic near the origin, and zero limiting outer term by its decay. Contracting~\eqref{eq:euclidean-stress-coefficient} with the frozen tensor in~\eqref{eq:stress-tangent-limit} proves (i). Every subsequence admits the preceding extraction, and the resulting scalar is independent of the affine frame; hence the limit holds along the full prescribed sequence.

	For part~(ii), Lemma~\ref{lem:shell-potential} and boundedness of $\Delta^Y v$ bound the normalized contribution of an inner dyadic shell by $C_as^2$, and that of an outer shell by $C_ar^2(r/s)^{\ell-4}$. Since $\ell=k+3\geq4$, summation and the separated-region bound give
	\[ \frac1{V_\epsilon(a)}\int_Y|\Delta^Y v|G_aW_{\epsilon,a}\,dm_Y\leq C_a\bigl[r^2(1+|\log r|)+\epsilon\bigr]\longrightarrow0. \]
	This proves~\eqref{eq:killing-vanishes}.
\end{proof}

\subsection{Completion of Part II}\label{sec:proof-density-equality}
\begin{proof}[Proof of Theorem~\ref{thm:log-density-identity}(ii)]
	Fix relatively compact open sets $A\Subset B\subset X$ and $f\in\Test(X)$ with $\supp f\Subset B$. Let $b_f\in L^2(A,m)$ be the representative constructed at the beginning of this section, and use the product construction above with poles over $A$. Fix a decreasing sequence $\epsilon_i\to0^+$ for which the finite-scale formulas hold on a common full-measure pole set. The soft-ball term satisfies
	\[ \frac1{V_{\epsilon_i}(a)} \int_{Y}H(\epsilon_i G_a)\Delta^Y v\,dm_Y \longrightarrow \Delta^Y v(a), \]
	because its support shrinks to $a$, its mass is comparable to $V_{\epsilon_i}(a)$, and $a$ is a Lebesgue point of $\Delta^Y v$. Proposition~\ref{prop:normalized-stress}, together with this limit, gives the limits of all three terms on the right-hand side of~\eqref{eq:general-finite-stress}. Consequently,
	\[ \langle\nabla \Phi_{\epsilon_i},\nabla f\rangle(x) \longrightarrow \frac{\ell-2}{2} \left[ \frac{2\tr(\Hess^Y v)(a)-\ell\Delta^Y v(a)}{\ell-2} +\Delta^Y v(a) \right] = \tr(\Hess^Y v)(a)-\Delta^Y v(a). \]
	On the pole region, where $\chi=1$,
	\[ \tr(\Hess^Y v)(a)-\Delta^Y v(a) = \tr(\Hess f)(x)-\Delta f(x). \]
	The directional derivative is independent of the Euclidean coordinate. Fubini therefore gives this convergence for $m$-almost every $x\in A$. The uniform local $L^2$ bound~\eqref{eq:local-regularized-directional} upgrades it to weak $L^2(A)$ convergence. By the construction and uniqueness of $b_f$ at the beginning of Section~\ref{sec:stress-limit}, the weak limit is $b_f$. Thus
	\[ b_f=\tr(\Hess f)-\Delta f \qquad m\text{-almost everywhere on }A. \]

	For every $\phi\in\Lip_c(A)$, the distributional representation of $b_f$ established above now gives
	\begin{align*}
		-\int_X\Phi\bigl[\phi\Delta f+\langle\nabla\phi,\nabla f\rangle\bigr]\,dm &=\int_A\phi b_f\,dm \\
		&=\int_X\phi\bigl[\tr(\Hess f)-\Delta f\bigr]\,dm.
	\end{align*}
	This proves~\eqref{eq:log-density-identity} for compactly supported global test functions.

	Let $f$ be any global test function and $\phi$ any compactly supported Lipschitz function. Choose a compact test cutoff equal to one near $\supp\phi$ and apply the preceding identity to its product with $f$. Locality identifies the gradient, Hessian, and Laplacian of this product with those of $f$ near $\supp\phi$. Hence~\eqref{eq:log-density-identity} holds in the full stated class.
\end{proof}

\subsection*{Acknowledgments}
The author thanks Professor Chong Song of Xiamen University for helpful comments on the volume-corrected heat entropy.   The author used ChatGPT (OpenAI) as a research aid to explore proof strategies, check calculations, and refine the language of the manuscript. The author retains full responsibility for the mathematical content of this paper.

\end{document}